\documentclass[12pt]{amsart}
\usepackage{amscd,amsmath,amsthm,amssymb,graphics}

\usepackage{amsmath}
\usepackage{amssymb}
\usepackage{amsbsy}
\usepackage{graphicx}
\usepackage{amsfonts}

\usepackage{amscd,amsmath,amsthm,amssymb}
\usepackage{pstcol,pst-plot,pst-3d}%\usepackage[T1]{fontenc}
\usepackage{lmodern,pst-node}%\usepackage{geometry}
\usepackage{pgf,tikz}

\usetikzlibrary{arrows}

\usetikzlibrary{arrows}

\def\NZQ{\Bbb}               % the font for N,Z,Q,R,C
\def\NN{{\NZQ N}}
\def\QQ{{\NZQ Q}}
\def\ZZ{{\NZQ Z}}

\def\frk{\frak}               % font for "Fraktur"

\def\mm{{\frk m}}

\def\Phi{{\frk n}}
\def\Phi{{\frk N}}
\def\MP{{\mathcal P}}

\def\MF{{\mathcal F}}

\def\MG{{\mathcal G}}
\def\MA{{\mathcal A}}
\def\MB{{\mathcal B}}
\def\e{{\epsilon}}

\def\opn#1#2{\def#1{\operatorname{#2}}} % to make operators
\opn\chara{char} \opn\length{\ell} \opn\pd{pd} \opn\rk{rk}
\opn\projdim{proj\,dim} \opn\injdim{inj\,dim} \opn\rank{rank}
\opn\depth{depth} \opn\grade{grade} \opn\height{height}
\opn\embdim{emb\,dim} \opn\codim{codim}

\opn\Tr{Tr} \opn\bigrank{big\,rank}
\opn\superheight{superheight}\opn\lcm{lcm}
\opn\trdeg{tr\,deg}%\emph{
\opn\reg{reg} \opn\lreg{lreg} \opn\ini{in} \opn\lpd{lpd}
\opn\size{size}\opn\bigsize{bigsize}
\opn\cosize{cosize}\opn\bigcosize{bigcosize}
\opn\sdepth{sdepth}\opn\sreg{sreg}
\opn\link{link}\opn\fdepth{fdepth}
\opn\index{index}
\opn\index{index}
\opn\indeg{indeg}
\opn\N{N}
\opn\SSC{SSC}
\opn\SC{SC}
\opn\Span{Span}
\opn\Der{Der}
\opn\div{div} \opn\Div{Div} \opn\cl{cl} \opn\Cl{Cl}
\opn\Spec{Spec} \opn\Supp{Supp} \opn\supp{supp} \opn\Sing{Sing}
\opn\Ass{Ass} \opn\Min{Min}\opn\Mon{Mon} \opn\dstab{dstab} \opn\astab{astab}
\opn\Syz{Syz}
\opn\reg{reg}
\opn\Ann{Ann} \opn\Rad{Rad} \opn\Soc{Soc}
\opn\Im{Im} \opn\Ker{Ker} \opn\Coker{Coker} \opn\Am{Am}
\opn\Hom{Hom} \opn\Tor{Tor} \opn\Ext{Ext} \opn\End{End}
\opn\Aut{Aut} \opn\id{id}

\opn\nat{nat}
\opn\pff{pf}%   \pf exists already
\opn\Pf{Pf} \opn\GL{GL} \opn\SL{SL} \opn\mod{mod} \opn\ord{ord}
\opn\Gin{Gin} \opn\Hilb{Hilb}\opn\sort{sort}
\opn\initial{init}
\opn\ende{end}
\opn\height{height}
\opn\type{type}
\opn\aff{aff} \opn\con{conv} \opn\relint{relint} \opn\st{st}
\opn\lk{lk} \opn\cn{cn} \opn\core{core} \opn\vol{vol}
\opn\link{link} \opn\star{star}\opn\lex{lex}\opn\Mon{Mon}\opn\Min{Min}
\opn\gr{gr}

\def\pot#1#2{#1[\kern-0.28ex[#2]\kern-0.28ex]}

\opn\dirlim{\underrightarrow{\lim}}
\opn\inivlim{\underleftarrow{\lim}}
\let\union=\cup
\let\sect=\cap

\let\tensor=\otimes
\let\iso=\cong
\let\Union=\bigcup

\let\Dirsum=\bigoplus

\let\to=\rightarrow
\let\To=\longrightarrow
\def\Implies{\ifmmode\Longrightarrow \else
        \unskip${}\Longrightarrow{}$\ignorespaces\fi}
\def\implies{\ifmmode\Rightarrow \else
        \unskip${}\Rightarrow{}$\ignorespaces\fi}
\def\iff{\ifmmode\Longleftrightarrow \else
        \unskip${}\Longleftrightarrow{}$\ignorespaces\fi}

\let\:=\colon
\newtheorem{Theorem}{Theorem}[section]
 \newtheorem{Lemma}[Theorem]{Lemma}
 \newtheorem{Corollary}[Theorem]{Corollary}
 \newtheorem{Proposition}[Theorem]{Proposition}
 \newtheorem{Remark}[Theorem]{Remark}
 
 \newtheorem{Example}[Theorem]{Example}
 
 \newtheorem{Definition}[Theorem]{Definition}

 \newtheorem{Discussion}[Theorem]{Discussion}

\def\Ac{{\mathcal A}}

\let\epsilon\varepsilon
\let\kappa=\varkappa
\def\qed{\ifhmode\textqed\fi
      \ifmmode\ifinner\quad\qedsymbol\else\dispqed\fi\fi}
\def\textqed{\unskip\nobreak\penalty50
       \hskip2em\hbox{}\nobreak\hfil\qedsymbol
       \parfillskip=0pt \finalhyphendemerits=0}
\def\dispqed{\rlap{\qquad\qedsymbol}}

\opn\dis{dis}
\def\pnt{{\raise0.5mm\hbox{\large\bf.}}}

\opn\Lex{Lex}

\begin{document}

 \title{Toric rings, inseparability and rigidity}

 \author {Mina Bigdeli, J\"urgen Herzog and Dancheng Lu}

\address{Mina Bigdeli, Faculty of Mathematics,  Institute for Advanced Studies in Basic Sciences (IASBS),
45195-1159 Zanjan,  Iran}
\email{mina.bigdeli@yahoo.com}

%\urladdr{http://secondauthorwebaddress}
%\thanks{The second author ... thanks}
%%%%%%%%

\address{J\"urgen Herzog, Fakult\"at f\"ur Mathematik, Universit\"at Duisburg-Essen,
45117 Essen, Germany}
\email{juergen.herzog@gmail.com}

\address{Dancheng Lu, Department of Mathematics, Soochow University, 215006 Suzhou, P.R.China}
\email{ludancheng@suda.edu.cn}

 \begin{abstract}  This article provides the basic algebraic background on infinitesimal deformations and presents the  proof of the well-known fact  that the non-trivial infinitesimal deformations of a $K$-algebra $R$ are parameterized by the elements of cotangent module $T^1(R)$ of $R$. In this article we focus on deformations of toric rings, and give an explicit description of $T^1(R)$ in  the case that $R$ is a toric ring.
 In particular, we are interested in unobstructed deformations which preserve the toric structure. Such deformations  we call separations. Toric rings which do  not admit any separation are called inseparable. We apply the theory to the edge ring of a finite graph. The coordinate ring of a convex polyomino may be viewed as the edge ring of a special class of bipartite graphs.  It is shown that the coordinate ring of any convex polyomino is inseparable.  We introduce the concept of semi-rigidity, and give a combinatorial description of the graphs whose edge ring is semi-rigid. The results are applied to show that  for $m-k=k=3$, $G_{k,m-k}$ is not rigid while for $m-k\geq k\geq 4$, $G_{k,m-k}$ is rigid. Here $G_{k,m-k}$ is the complete bipartite graph $K_{m-k,k}$ with one edge removed.

 \end{abstract}



\thanks{Part of this article  was written while the third author was visiting the Department of
Mathematics of University Duisburg-Essen. He wants to express his thanks for the hospitality}

 \maketitle

 \setcounter{tocdepth}{1}
\tableofcontents

\section*{Introduction}

In this paper we study infinitesimal deformations and  unobstructed deformations of toric rings which preserve the toric structure, and apply this theory to edge rings of bipartite graphs. Already in \cite{A1} and \cite{A2}, infinitesimal and homogeneous deformations of toric varieties have been considered  from a geometric point of view. The viewpoint of this paper is more algebraic and  does not exclude non-normal toric rings, having in mind toric rings which naturally appear in combinatorial contexts. This aspect of deformation theory has also been pursued in the papers \cite{AC2},\cite{AC1} and \cite{ABHL}, where deformations of Stanley-Reisner rings attached to simplicial complexes were studied.

Due to the lack of a suitable reference in which the basics of deformation theory are presented in algebraic terms, we give in the first two sections a short introduction to deformation theory.

Let $K$ be a field. It will be shown that  the infinitesimal deformations  of a finitely generated  $K$-algebra $R$ are parameterized by the elements of the cotangent module $T^1(R)$, which in the case that $R$ is a domain is isomorphic to $\Ext_R^1(\Omega_{R/K}, R)$. Here $\Omega_{R/K}$ denotes  the module of differentials of $R$ over $K$. The ring $R$ is called {\em rigid} if $T^1(R)=0$. We  refer the reader to \cite{SBook} for a further study  of deformation theory.

In this article we focus on deformations of  toric algebras. They may be viewed as affine semigroup rings.
Let $H$ be an affine semigroup and $K[H]$ its affine semigroup ring. We are interested in the module $T^1(K[H])$. This module is naturally $\ZZ H$-graded. Here $\ZZ H$ denotes the associated group of $H$ which for an affine semigroup is a free group of finite rank. For each $a\in \ZZ H$, the $a$-graded component $T^1(K[H])_a$   of $T^1(K[H])$  is a finite dimensional $K$-vector space.

In Section~3 we describe the vector space $T^1(K[H])_a$ and provide a method   to compute its dimension. Let $H\subset  \ZZ^m$ with generators $h_1,\ldots, h_n$. Then  $K[H]$  is the $K$-subalgebra of the ring $K[t_1^{\pm 1},\ldots,t_m^{\pm 1}]$ of Laurent polynomials generated by the monomials $t^{h_1}, \ldots, t^{h_n}$. Here $t^a=t_1^{a(1)}\cdots t_m^{a(m)}$ for $a=(a(1),\ldots,a(m))\in \ZZ^m$. Let $S=K[x_1,\ldots,x_n]$ be the polynomial ring over $K$ in the indeterminates $x_1,\ldots,x_n$. Then $S$ may be viewed as a $\ZZ H$-graded ring with $\deg x_i=h_i$, and the $K$-algebra homomorphism $S\to K[H]$ with $x_i\mapsto t^{h_i}$ is a homomorphism of $\ZZ H$-graded $K$-algebras. We denote by  $I_H$  the kernel of this homomorphism. The ideal $I_H$ is called the toric ideal associated with $H$. It is generated by homogeneous binomials. To describe these binomials, consider the group homomorphism $\ZZ^n\to \ZZ^m$ with $\epsilon_i\mapsto h_i$, where $\epsilon_1,\ldots,\epsilon_n$ is the canonical basis of $\ZZ^n$. The kernel $L$ of this group homomorphism is a lattice of $\ZZ^n$ and is called the {\em relation lattice} of $H$. Here a lattice just means a subgroup of $\ZZ^n$.  For $v=(v(1),\ldots,v(n))\in \ZZ^n$ we define the binomial $f_v=f_{v^+}-f_{v^-}$ with $f_{v^+}=\prod_{i,\; v(i)\geq 0}x_i^{v(i)}$ and $f_{v^-}=\prod_{i,\; v(i)\leq 0}x_i^{-v(i)}$, and let $I_L$ be the ideal generated by the binomials $f_v$ with $v\in L$. It is well known that $I_H=I_L$. Each $f_v\in I_H$ is homogeneous of degree $h(v)=\sum_{i,\; v(i)\geq 0}v(i)h_i$. Let $f_{v_1},\ldots,f_{v_s}$ be a system of generators of $I_H$. We  consider the $(s\times n)$-matrix
\[
A_H= \left( \begin{array}{cccc}
v_1(1) & v_1(2) & \ldots & v_1(n) \\
v_2(1) & v_2(2) & \ldots & v_2(n)\\
\vdots  & \vdots  &  & \vdots \\
v_s(1) & v_s(2) & \ldots & v_s(n)
\end{array} \right).
\]
Summarizing the results  of Section~3, for any $a\in \ZZ H$ the $K$-dimension of $T^1(K[H])_a$  can be computed as follows: let $l=\rank A_H$, $l_a$ be the rank of the submatrix of $A_H$ whose rows  are  the $i$th rows of $A_H$ for which $a+h(v_i)\not \in H$, and let $d_a$  be the rank of the submatrix of $A_H$ whose columns  are  the $j$th columns of $A_H$ for which $a+h_j\in H$. Then
\[
\dim_K T^1(K[H])_a=l-l_a-d_a.
\]

In Section~4 we introduce the concept of separation for a torsionfree lattice $L\subset \ZZ^n$. Note that a lattice $L\subset \ZZ^n$ is torsionfree if and only if it is the relation lattice of some affine semigroup. Given an integer $i\in [n]=\{1,2,\ldots,n\}$, we say that $L$ admits an $i$-{\em separation} if there exists a torsionfree lattice $L'\subset \ZZ^{n+1}$ of the same rank as $L$ such that $\pi_i(I_{L'})=I_L$, where $\pi_i\: S[x_{n+1}]\to S$ is the $K$-algebra homomorphism which identifies $x_{n+1}$ with $x_i$. An additional condition makes sure that this deformation which induces an element in $T^1(K[H])_{-h_i}$ is non-trivial, see \ref{separable} for the precise definition. We say that $L$ is {\em inseparable}, if for all $i$, the lattice $L$ admits no $i$-separation, and we call $H$ and its toric ring {\em inseparable} if its relation  lattice is inseparable.
In particular, if the generators of $H$ belong to a hyperplane of $\ZZ^m$, so that $K[H]$ also admits a natural standard grading, then $H$ is inseparable if $T^1(K[H])_{-1}=0$, see Theorem~\ref{sufficient}. In general, the converse is not true since the infinitesimal deformations given by non-zero elements of $T^1(K[H])_{-1}$ may be obstructed. We demonstrate this theory and show that a numerical semigroup generated by three elements  which is not a complete intersection is $i$-separable for $i=1,2,3$, while if it is a complete intersection it is $i$-separable for at least two $i\in \{1,2,3\}$. For the proof of this fact we use the structure theorem of such semigroup rings given in \cite{H}.

Section~5 is devoted to the study of $T^1(R)$ when $R$ is the edge ring of a bipartite graph. This class of rings has been well studied in combinatorial commutative algebra, see e.g. \cite{OH} and \cite{V}. For a given simple graph $G$ of the vertex set $[n]$  one considers the edge ring $R=K[G]$ which is the toric ring generated over $K$ by the monomials $t_it_j$ for which $\{i,j\}$ is an edge of $G$.  Viewing the edge ring as a  semigroup ring $K[H]$, the edges $e_i$ of $G$ correspond  the generators $h_i$ of the semigroup $H$. We say that $G$ is {\em inseparable} if the corresponding semigroup is  inseparable. The main result  of this section is a combinatorial criterion for a bipartite graph $G$ to be inseparable.  Let $C$ be a cycle of $G$ and  $e$ a chord of $G$. Then $e$ splits $C$ into two disjoint connected components  $C_1$ and $C_2$ which are obtained by restricting $C$ to the complement of $e$.  A path $P$ of $G$ is called a {\em crossing path} of $C$ with respect to $e$ if one end of $P$ belongs to $C_1$ and the other end to $C_2$. Now the criterion (Corollary~\ref{unique}) says that a bipartite graph $G$ is inseparable if and only if for any cycle $C$ which has a unique chord $e$, there exists a crossing path  of $C$ with respect to $e$. In particular, if no cycle  has a chord, then $G$ is inseparable. By using this criterion we show in Theorem~\ref{polyomino} that  the coordinate ring of any convex polyomino, which may be interpreted as a special class of edge rings, is inseparable.

The concept of semi-rigidity is introduced in Section 6. We call $H$ {\em semi-rigid} if $T^1(K[H])_{-a}=0$  for all $a\in H$, and characterize in Theorem~\ref{semi-rigid} the semi-rigidity of bipartite graphs  in terms of the non-existence of  certain constellations of edges and cycles of the graph. The classification of  rigid  bipartite graphs  is much more complicated, and we do not have a general combinatorial criterion to see when a bipartite graph is rigid. However we study, as an example, a particular class of bipartite graphs in Section 7.  For   $m-k\geq k\geq 3$,  we consider  the graph $G_{k,m-k}$  which is obtained by removing an edge from the complete bipartite graph $K_{m-n,n}$.  It is shown in Proposition~\ref{main4} that for $m-k=k=3$, $G_{k,m-k}$ is not rigid while for $m-k\geq k\geq 4$, $G_{m-k,k}$ is rigid. It remains a challenging open problem to classify all rigid bipartite graphs.

\section{Infinitesimal  deformations}

In this section we give a short introduction to infinitesimal deformations. We fix a field $K$ and let $\Ac$ be the category of standard graded $K$-algebras with homogeneous homomorphisms of degree zero as its morphisms. For each $A\in \Ac$ we denote by $\mm_A$ the graded maximal ideal of $A$.

Let $A\in \Ac$. A  {\em deformation} of $A$ with basis $B$ is a flat homomorphism $B\to C$ of standard graded $K$-algebras whose  fiber $C/\mm_BC$ is isomorphic to $A$ as  $K$-algebra.

Thus we obtain  a commutative diagram of standard graded $K$-algebras
\[
\begin{CD}
C @>>> A\\
@AAA    @AAA\\
B @>>> K.
\end{CD}
\]

Let $I\subset B$ be a graded ideal. Then $B \to C$ induces the  flat homomorphism $B/I\to C/IC$, and hence induces the deformation

\[
\begin{CD}
C/IC @>>> A\\
@AAA    @AAA\\
B/I @>>> K.
\end{CD}
\]

\medskip
We denote by $K[\epsilon]$ the $K$-algebra with $\epsilon\neq 0$ but $\epsilon^2=0$. In other words, $K[\epsilon]=K[x]/(x^2)$.

Any surjective $K$-algebra homomorphism $B\to K[\epsilon]$ induces a deformation  of $A$ with basis $K[\epsilon]$. A deformation of $A$ with basis $K[\epsilon]$ is called an {\em infinitesimal deformation}.
\[
\begin{CD}
C @>>> A\\
@AAA    @AAA\\
K[\epsilon] @>>> K.
\end{CD}
\]

\begin{Lemma}
\label{flatness} $K[\epsilon]\to C$ is flat if and only if $0:_C\epsilon=\epsilon C$.
\end{Lemma}

\begin{proof} It is known that $C$ is a flat $K[\epsilon]$-module, if and only if $$\Tor_1^{K[\epsilon]}(C,K[\epsilon]/(\epsilon))=0.$$

We have the exact sequence
\[
\begin{CD}
\cdots @> \epsilon >> K[\epsilon]  @> \epsilon >> K[\epsilon] @>>>  K[\epsilon]/(\epsilon)@>>> 0.
\end{CD}
\]

Tensoring it with $C$ we obtain the complex
\[
\begin{CD}
\cdots @> \epsilon >> C @> \epsilon >> C @>>> 0,
\end{CD}
\]

whose $i$th homology is  $\Tor_i^{K[\epsilon]}(C,K[\epsilon]/(\epsilon))$.

Thus we see that $\Tor_1(C,K[\epsilon])/(\epsilon))=(0:_C\epsilon)/\epsilon C$. The assertion follows.
\end{proof}

Whenever there is a deformation $B\to C$ of $A$ with $B\neq K$, then there is also an infinitesimal deformation, induced  by  a surjective $K$-algebra homomorphism $B\to K[\epsilon]$.
An infinitesimal deformation always exists. For example
\[
\begin{CD}
A[\epsilon]=A\otimes_K K[\epsilon] @>>> A\\
@AAA    @AAA\\
K[\epsilon] @>>> K.
\end{CD}
\]

However this is a trivial deformation. More generally we say that $C$ is a {\em trivial deformation} of $A$ with basis $B$, if there exists an  isomorphism $C\to A\otimes_K B$ such that  the diagram

\[
\begin{array}{ccccc}
&&C\\
&\nearrow &\downarrow & \searrow\\
B&\to& A\tensor B&\to &A
\end{array}
\]
is commutative. Here $ A\tensor B\to A$ is the composition of $ A\tensor B\to A\tensor B/\mm_B$ and  $A\tensor B/\mm_B\cong A$.

The algebra $A$ is called {\em rigid}, if it admits no non-trivial infinitesimal deformation.

Can an infinitesimal deformation of $A$ be lifted to a deformation with  basis $B$? In general there are obstructions to do this.

An infinitesimal deformation of $A$ which is induced by  a deformation of $A$ with basis $K[t]$ (the polynomial ring), is called   {\em unobstructed}.

\section{The cotangent functor $T^1$}

How can we find and classify all non-trivial infinitesimal deformations of $S/I$?

Let   {\em $S=K[x_1,\ldots,x_n]$} be the polynomial ring and let  $A=S/I$, where $I\subset S$ is a graded ideal.

Let  {\em $J\subset S[\epsilon]$} be  a graded  ideal, and let    $C=S[\epsilon]/J$ such that $C/\epsilon C=S/I$.

\begin{Proposition} Let   $I=(f_1,\ldots,f_m)$.  Then $J=(f_1+g_1\epsilon,\ldots, f_m+g_m\epsilon)$  and  $K[\epsilon]\to S[\epsilon]/J$ is flat if and only  $\varphi: I\to S/I$   with  $f_i\mapsto g_i+I$
is a well-defined $S$-module homomorphism.
\end{Proposition}

\begin{proof} Assume that  $K[\epsilon]\to C$ is flat. Let  $\sum_ih_if_i=0$ with $h_i\in S$. We want to show that  $\sum_ih_ig_i\in I$,  because this is equivalent to saying that $\varphi$ is well-defined.  To see this, let  $g=\sum_ih_i(f_i+\epsilon g_i)$.
Then  $g=\epsilon(\sum_ih_ig_i)$ and  $g\in J$. Therefore, $\sum_ih_ig_i\in J:\epsilon$. Since $C$ is a flat  $K[\epsilon]$-module, there exists  $p\in S$ such that $\sum_ih_ig_i-\epsilon p\in J$. Modulo  $\epsilon$ it follows that  $\sum_ih_ig_i\in I$.

Conversely, we want to show that $K[\epsilon]\to S[\epsilon]/J$ is flat. By Lemma~\ref{flatness},  we must show that $J:\epsilon=\epsilon S+J$. It suffices to prove that $J:\epsilon\subset \epsilon S+J$, because the other inclusion is trivial. Now let $g\in J:\epsilon$, where $g=a+\epsilon b$ with $a,b\in S$. Then
\[
\epsilon a = \epsilon g = \sum_{i=1}^m (h_i+\epsilon h_i')(f_i+\epsilon g_i)
\]
for some $h_i$ and $h_i'$ in $S$.

It follows that $\sum _{i=1}^mh_if_i=0$, and  that
$a=\sum_{i=1}^mh_ig_i+\sum_{i=1}^mh_i'f_i$.
Our assumption implies that $\sum_{i=1}^mh_ig_i\in I$. Therefore, $a\in I$. Let $a=\sum_{i=1}^ma_if_i$. Then $a=\sum_{i=1}^ma_i(f_i+\epsilon g_i)-\epsilon \sum_{i=1}^ma_ig_i$. Hence, $a\in \epsilon S+J$ and therefore  also $g\in \epsilon S+J$.
\end{proof}

The above proposition says that the infinitesimal deformations of $S/I$ are in bijection to the elements of  $I^*:=\Hom_S(I,S/I)$.

Let $C=S[\epsilon]/J$ be an infinitesimal deformation of $S/I$. Then this deformation  is trivial if and only if there is a $K[\epsilon]$-automorphism $\varphi : S[\epsilon] \to S[\epsilon]$ which is the identity map on $S$ modulo $\epsilon$ and such that $\varphi(IS[\e])=J$.

\medskip
Let  $\Der_K(S)$ be the set of  $K$-derivations $\partial : S\to S$ of $S$.   Recall that a $K$-linear map  $\partial : S\to S$ is called a {\em $K$-derivation}, if
\begin{enumerate}
\item[(i)] $\partial(a)=0$ if $a\in K$,
\item[(ii)] $\partial(fg)=f\partial(g)+g\partial(f)$ for all $f,g\in S$.
\end{enumerate}

If $\partial, \partial'$ are $K$-derivations and $s,s'\in S$, then $s\partial+s'\partial'$ with $(s\partial+s'\partial')(f):=s\partial(f)+s'\partial'(f)$ for all $f\in  S$ is again a $K$-derivation. Thus $\Der_K(S)$ is an $S$-module.

Examples of $K$-derivatives are  the partial derivatives $\partial_i$ which are defined by the property that
$\partial_i(x_j)=1$ if $j=i$ and $\partial_i(x_j)=0$, if $j\neq i$. It is known that  $\Der_K(S)$ is a free $S$-module with basis $\partial_1,\ldots,\partial_n$

\begin{Proposition}
 \label{der}
 The infinitesimal deformation $S[\e]/J$ of $S/I$ is trivial if  and only if there exists $\partial\in \Der_K(S)$ such that  $J=(f_1+\partial(f_1)\epsilon, \ldots, f_m+\partial(f_m)\epsilon)$.
\end{Proposition}

\begin{proof} Suppose there exists $\partial\in \mbox{Der}_K(S)$ with $J=(f_1+\partial(f_1)\epsilon, \ldots, f_m+\partial(f_m)\epsilon)$.
We define the $K[\epsilon]$-algebra automorphism $\varphi: S[\epsilon]\to S[\epsilon]$ with $x_i\mapsto x_i+\partial(x_i)\epsilon$.

Then \begin{eqnarray*}
\varphi(\prod_{i=1}^nx_i^{a_i})&=&\prod_{i=1}^n(x_i+\partial (x_i)\epsilon)^{a_i}=\prod_{i=1}^n(x_i^{a_i}+a_ix_i^{a_i-1}\partial(x_i)\epsilon)\\
&=&\prod_{i=1}^nx_i^{a_i}+\sum_{i=1}^na_ix_i^{a_i-1}\partial(x_i)\epsilon \prod_{j\neq i}x_j^{a_j}\\
&=& \prod_{i=1}^nx_i^{a_i}+\partial(\prod_{i=1}^nx_i^{a_i})\epsilon.
\end{eqnarray*}
Since $\varphi $ and $\partial$ are $K$-linear, it follows that  $\varphi(f_i)=f_i+\partial(f_i)\epsilon$ for all $i$.
Therefore, $\varphi(IS[\epsilon])=J$.

Conversely, suppose $J=(f_1+g_1\e,\ldots,f_m+g_m\e)$ and that there exists a $K[\e]$-isomorphism $\varphi\:S[\e]\to S[\e]$ with $\varphi(x_i) =x_i+c_i\e$ for $i=1,\ldots,n$ and such that $\varphi(f_j)=f_j+g_j\e$ for $j=1,\ldots,m$. Let $\partial$ be the $K$-derivation with $\partial(x_i)=c_i$. A calculation  as before shows that $g_j=\partial(f_j)$ for $j=1,\ldots,m$.
\end{proof}

As a consequence of our considerations so far, we see the following:  if we consider the natural map $\delta^*: \Der(S)_K\to I^*$ which assigns to $\partial \in \Der_K(S)$ the element $\delta^*(\partial)$ with
$$\delta^*(\partial)(f_i)=\partial f_i+I,$$
then  the non-zero elements of  $\Coker \delta^*$ are in bijection to the  non-trivial infinitesimal deformations of $S/I$.
This cokernel is denoted by $T^1(S/I)$ and is called the  {\em first cotangent module} of $S/I$.

For any $B$-algebra homomorphism  $B\to A$ and any $A$-module $M$,  there exist modules  $T^i(A/B, M)$ and  $T_i(A/B, M)$ for $i=0,1,\ldots$, the so-called  {\em tangent} and {\em cotangent modules}. They are functorial  in all three variables.

In  1967,  Lichtenbaum and  Schlessinger \cite{LS} first introduced the functors  $T^i$ for $i=0,1,2$ in the paper  ``On the cotangent complex of a morphism" Trans AMS.

Quillen \cite{Qu}  in 1970  and Andr\'{e} \cite{An} in  1974 defined the higher cotangent functors and developed their theory.

In characteristic $0$,  a different and simpler approach is given by Palamodov \cite{P} by using DGA algebras.

\medskip
$T^1(S/I)$ is a finitely generated graded (multigraded) $S$-module if $S/I$ is graded (multigraded).
Furthermore, $S/I$ is rigid if $S/I$ admits no non-trivial infinitesimal deformations, and this is the case  if and only if  $T^1(S/I)=0$.

\begin{Example} {\em Let $I=(xy,xz,yz)\subset S=K[x,y,z]$, and $L=(xw,xz,yz)\subset T=K[x,y,z,w]$.
Then $t:=w-y$ is a non-zerodivisor of $T/L$. Thus $K[t]\to T/L$ is flat, and hence $T/L\otimes K[\epsilon]$ with $K[\epsilon]=K[t]/(t^2)$ is an infinitesimal deformation of $S/I$.
Note that  $T=K[x,y,z,t]$ and $L=(xy+xt,xz,yz)$. Hence $T/L\otimes K[\epsilon]\cong S[\epsilon]/(xy+x\epsilon,xz, yz)$.

We claim that  $S[\epsilon]/(xy+x\epsilon,xz, yz)$ is a non-trivial deformation of $S/I$.
Suppose it is trivial.  Then there exists  $\partial\in \mbox{Der}_K(S)$ with $\partial(xy)=x$ and   $\partial(xz)=\partial(yz)=0$.

The module  $\Der_K(S)$ is a free  $S$-module with basis  $\partial_x,\partial_y,\partial_z$.
Let  $\partial =f\partial_x+g\partial_y+h\partial_z$. Since $\partial(xz)=0=\partial(yz)$, we conclude that $fy=gx$. Thus, $f=xr, g=yr$ with $r\in S$. The condition that $x=\partial(xy)$ implies that $fy+gx=x$. Hence $2yr=1$, a contradiction.
The calculations also show that $T^1(S/I)_{-1}\neq 0$. We refer readers  to \cite{ABHL} for the details on infinitesimal deformations of  squarefree monomial ideals. }
\end{Example}

Let  $R=S/I$, where  $I\subset S$ is  a graded ideal, and let  $M$ be  a graded  $R$-module.

A $K$-derivation  $\delta : R\to M$  is a $K$-linear map such that
\[
\delta(rs)=r\delta(s)+s\delta(r) \quad \text{for all $r,s\in R$.}
\]

The {\em module of differentials} $\Omega_{R/K}$ is defined by the universal property that there exists a $K$-derivation
$d:R\to \Omega_{R/K}$ such that for any derivation $\delta: R\to M$ there exists an $R$-module homomorphism $\varphi : \Omega_{R/K}\to M$ such that
\[
\partial =\varphi\circ d.
\]

Let $I=(f_1,\ldots,f_m)$. Then

\[
\Omega_{R/K}\cong (\bigoplus_{i=1}^nRdx_i)/U,
\]
where $\bigoplus_{i=1}^nRdx_i$ is the free $R$-module with basis $dx_1,\ldots,dx_n$ and $U$ is generated by the elements $\sum_{i=1}^n\overline{\partial_i (f_j)}dx_i$ for $j=1,\ldots,m$,  and where $\bar{g}$ denotes the residue class of a polynomial $g\in S$ modulo $I$.
Thus the relation matrix of $\Omega_{R/K}$ is the Jacobian matrix.

\medskip
There is the fundamental exact sequence of $R$-modules
\[
I/I^2\to \bigoplus_{i=1}^nRdx_i\to \Omega_{R/K}\to 0,
\]
where $\delta:I/I^2\to  \bigoplus_{i=1}^nRdx_i$ is the $R$-linear map
\[
f +I^2\mapsto \sum_{i=1}^n\overline{\partial_i (f)}dx_i.
\]

For an $R$-module $M$ we use $M^*$ to denote the dual module $\Hom_R(M,R)$ of $M$. By dualizing  the  fundamental exact sequence  one obtains the exact sequence

\[
\delta^*: \bigoplus_{i=1}^nR\partial_i\to (I/I^2)^*\to T^1(R)\to 0.
\]

In general, the map $\delta:I/I^2\to  \bigoplus_{i=1}^nRdx_i$ is not injective.
Let $V=\Ker \delta$. If  $R$ is reduced and  $K$ is a perfect field,  then $\Supp V\sect\Ass R=\emptyset$. To see this, we first observe that $\Ass R= \Min R$, where $\Min R$ denotes  the set of minimal prime ideals of $R$. Let $\wp \in \Min R$ and $P=\pi^{-1}(\wp)\subset S$, where $\pi\:S\to R$ is the canonical epimorphism. Then $IS_P=PS_P$ and $R_\wp\iso S_P/PS_P=L$ is a field. Since $K$ is perfect it follows that $L/K$ is a separable extension. Therefore, by \cite[Corollary 6.5]{Ku}  the natural map $\sigma\: PS_P/P^2S_P\to \Omega_{S_P/K}\tensor S_P/PS_P$ is injective. Since $(I/I^2)_P=PS_P/P^2S_P$ and since the module of differential localizes we also have
\[
(\bigoplus_{i=1}^nRdx_i)_P=(\Omega_{S/K}\tensor S/I)_P=\Omega_{S_P/K}\tensor S_P/PS_P.
\]
This  shows that $\sigma=\delta\tensor S_P$. Thus  $V_P=0$, as desired.

Now as we know that  $\Supp V\sect\Ass R=\emptyset$,  it follows that  $V^*=\Hom_R(V,R)=0$.
Therefore, by dualizing the exact sequence
\[
0\to V\to I/I^2\to U\to 0
\]
we obtain that $U^*=(I/I^2)^*$.

Now the fundamental exact sequence yields
\begin{eqnarray*}
\Ext^1_R(\Omega_{R/K},R)&=&\Coker(\bigoplus_{i=1}^nR\partial_i\to U^*)\\
&=&\Coker(\bigoplus_{i=1}^nR\partial_i\to (I/I^2)^*)=T^1(R).
\end{eqnarray*}

\section{$T^1$  for toric rings}
Let $H$ be an {\em affine semigroup}, that is, a finitely generated subsemigroup of $\ZZ^m$ for some $m>0$. Let $h_1,\ldots,h_n$ be the minimal generators of $H$, and fix a field $K$ of characteristic $0$. The toric ring $K[H]$ associated with $H$  is the $K$-subalgebra of the ring $K[t_1^{\pm 1},\ldots,t_m^{\pm 1}]$ of Laurent polynomials generated by the monomials $t^{h_1}, \ldots, t^{h_n}$. Here $t^a=t_1^{a(1)}\cdots t_m^{a(m)}$ for $a=(a(1),\ldots,a(m))\in \ZZ^m$.

Let $S=K[x_1,\ldots,x_n]$ be the polynomial ring over $K$ in the variables $x_1,\ldots,x_n$. The $K$-algebra $R=K[H]$ has a presentation $S\to R$ with $x_i\mapsto t^{h_i}$ for $i=1,\ldots,n$. The kernel $I_H\subset S$ of this map is called the {\em toric ideal} attached to $H$. Corresponding to this presentation of $K[H]$ there is  a presentation $\NN^n\to H$ of $H$ which can be extended to the group homomorphism $\ZZ^n\to \ZZ^m$ with $\epsilon_i\mapsto h_i$ for $i=1,\ldots,n$, where $\epsilon_1,\ldots,\epsilon_n$ denotes the canonical basis of $\ZZ^n$. Let $L\subset \ZZ^n$ be the kernel of this group homomorphism. The lattice $L$ is called the {\em relation lattice} of $H$. Note that  $L$ is a free abelian group and $\ZZ^n/L$ is torsion-free.

For a vector $v\in \ZZ^n$ with $v=(v(1),\ldots,v(n))$, we set
\[
v_+=\sum_{i,\; v(i)\geq 0}v(i)\epsilon_i\quad  \text{and}\quad  v_-=\sum_{i,\; v(i)\leq 0}-v(i)\epsilon_i.
\]
Then $v=v_+-v_-$. It is a basic fact and well-known (see e.g. \cite[Lemma 5.2]{EH})  that $I_H$ is generated by the binomials $f_v$ with $v\in L$, where $f_v=x^{v_+}-x^{v_-}$.

We define an $H$-grading on $S$ by setting $\deg x_i=h_i$. Then $I_H$ is a graded ideal with $\deg f_v= h(v)$, where
\begin{eqnarray}
\label{h(v)}
h(v)=\sum_{i,\; v(i)\geq 0}v(i)h_i\;  (=\sum_{i,\; v(i)\leq 0}-v(i)h_i).
\end{eqnarray}

Let $v_1,\ldots ,v_r$ be a basis of $L$. Since $I_H$ is a prime ideal we may localize $S$ with respect to this prime ideal and obtain
\[
I_HS_{I_H}=(f_{v_1},\ldots,f_{v_r})S_{I_H}.
\]
In particular, we see that
\begin{eqnarray}
\label{height=rank}
\height I_H=\rank L.
\end{eqnarray}

\medskip
Let, as before,  $\Omega_{R/K}$ be the module of differentials of $R$ over $K$. Since $R$ is a domain and $\chara(K)=0$, the cotangent module $T^1(R)$ is isomorphic to $\Ext^1_R(\Omega_{R/K}, R)$, and  since $R$ is $H$-graded it follows that $\Omega_{R/K}$ is $H$-graded as well, and hence $\Ext^1_R(\Omega_{R/K}, R)$ and $T^1(R)$ are $\ZZ H$-graded. Here $\ZZ H$ denotes the associated group of $H$, that is, the smallest subgroup of $\ZZ^m$ containing $H$. It is our goal to compute the graded components $T^1(R)_a$ of $T^1(R)$ for $a\in \ZZ H.$

The module of differentials has a presentation
\[
\Omega_{R/K}=(\Dirsum_{i=1}^nRdx_i)/U,
\]
where $U$ is the submodule of the free $R$-module $\Dirsum_{i=1}^nRdx_i$  generated by the elements
$df_v$ with $v\in L$, where
\[
df_v=\sum_{i=1}^n(\partial f_v/\partial x_i) dx_i.
\]
Here $\partial f_v/\partial x_i$ stands for partial derivative of $f_v$ with respect to $x_i$, evaluated modulo $I_H$.

One verifies at once that
\begin{eqnarray}
df_v=\sum_{i=1}^nv(i)t^{h(v)-h_i}dx_i.
\end{eqnarray}

For $i\in [n]$, the basis element $dx_i$ of $\Omega_{S/K}\tensor_S R=\Dirsum_{i=1}^nRdx_i$ is given the degree $h_i$. Then $U$ is an $H$-graded submodule of $\Omega_{S/K}\tensor_S R$, and $\deg df_v= \deg f_v=h(v)$.

\medskip
For any $\ZZ H$-graded $R$-module $M$ we denote by $M^*$ the graded $R$-dual $\Hom_R(M,R)$. Then the exact sequence of $H$-graded $R$-modules
\[
0\rightarrow U\rightarrow \Omega_{S/K}\otimes_S R\rightarrow \Omega_{R/K}\rightarrow 0
\]
gives rise to the exact sequence
\[
(\Omega_{S/K}\otimes_S R)^*\to U^* \to T^1(R)\to 0
\]
of $\ZZ H$-graded modules.  This exact sequence may serve as the definition of $T^1(R)$, namely, to be  the cokernel of $(\Omega_{S/K}\otimes_S R)^*\to U^*$.

Let $f_{v_1},\ldots, f_{v_s}$ be a system of generators of $I_H$, where we may assume that for $r\leq s$, the elements $v_1,\ldots,v_r$ form a  basis of $L$. In general $s$ is much larger than $r$. Observe that the elements $df_{v_1},\ldots,df_{v_s}$ form a system of generators of $U$.

We let $F$ be a free graded $R$-module with basis $g_1,\ldots,g_s$ such that $\deg g_i=\deg df_{v_i}$ for $i=1,\ldots,s$, and define the $R$-module epimorphism $F\to U$ by $g_i\mapsto df_{v_i}$ for $i=1,\ldots,s$. The kernel of $F\to U$ we denote by $C$. The composition $F\to  \Omega_{S/K}\otimes_SR$ of the epimorphism $F\to U$ with the inclusion map $U\to  \Omega_{S/K}\otimes_SR$ will be denoted by $\delta$. We identify $U^*\subset F^*$ with its image in $F^*$. Then $T^1(R)=U^*/\Im \delta^*$ and $U^*$ is the submodule of $F^*$ consisting of all $\varphi\in F^*$ with $\varphi(C)=0$.

We first describe the $\ZZ H$-graded components of $U^*$. Let $a\in \ZZ H$. We denote by $KL$  the $K$-subspace of $K^n$ spanned  by $v_1,\ldots,v_s$  and by $KL_a$ the $K$-subspace of $KL$ spanned by the vectors  $v_i$ with $i\notin \MF_a$. Here the set $\MF_a$ is defined to be
\[
\MF_a=\{i\in[s]\: a+h(v_i)\in H\}.
\]
 Then we have

\begin{Theorem} \label{F1}
For all $a\in \ZZ H$, we have
\[
\dim_K (U^*)_a=\dim_K KL-\dim_K KL_a.
\]
\end{Theorem}

\begin{proof} Let $\sigma_1,\ldots ,\sigma_s $ be the canonical basis of $K^s$ and $W\subset K^s$ be the kernel of the $K$-linear map $K^s\to KL$ with $\sigma_i\mapsto v_i$ for $i=1,\ldots, s$.

We will show that
\begin{eqnarray}
\label{equal}
\hspace{1cm}(U^*)_a\iso \{\mu\in K^s\:\; \text{$\mu(i)=0$  for $i\in [s]\setminus \MF_a$ and $\langle \mu,  \lambda\rangle=0$  for all $\lambda\in W$}\},
\end{eqnarray}
as $K$-vector space.

Assuming this isomorphism has been proved, let $X_a$ be the image of $W\subset K^s$ under the canonical projection $K^s\to V_a=\Dirsum_{i\in \MF_a }K\sigma_i$. Then (\ref{equal}) implies  that $(U^*)_a$ is isomorphic to the orthogonal complement of $X_a$ in $V_a$.  Thus,
\begin{eqnarray}
\label{step}
\dim_K(U^*)_a=|\MF_a|-\dim_KX_a.
\end{eqnarray}
Let $Z_a=\Dirsum_{i\not\in\MF_a}K\sigma_i$ and $Y_a$ the cokernel of $X_a\to V_a$. Then we obtain a commutative diagram with exact rows and columns
\begin{eqnarray*}
\label{diagram}
\begin{CD}
&&& 0 &  &&   & 0 &  &&   & 0 &\\
&&& @VVV& & @VVV &  & @VVV & \\
0 @>>> &W\sect Z_a& @>>> & Z_a &@>>> &KL_a& @>>> 0\\
&&& @VVV& & @VVV &  & @VVV & \\
0 @>>> &W& @>>> & K^s&@>>> &KL& @>>> 0\\
&&& @VVV& & @VVV &  & @VVV & \\
0 @>>> &X_a& @>>> & V_a&@>>> &Y_a& @>>> 0\\
&&& @VVV& & @VVV &  & @VVV & \\
&&& 0 &  &&   & 0 &  &&   & 0 &\\
\end{CD}
\end{eqnarray*}
Now (\ref{step}) implies that $\dim_K(U^*)_a=\dim_K Y_a$, and the diagram shows that $\dim_K Y_a =\dim_K KL-\dim_K KL_a$.

\medskip
It remains to prove the isomorphism (\ref{equal}). Observe that $(U^*)_a=\{\varphi\in (F^*)_a\:\; \varphi(C)=0\}$, where $C$ is the kernel of $F\to U$. Let $\varphi\in (F^*)_a$. Then  $\varphi= \sum_{i=1}^s\varphi(g_i)g_i^*$, where $g_1^*,\ldots,g_s^*$ is the basis of $F^*$ dual to $g_1,\ldots,g_s$.

Since $\deg g_i^*=-\deg df_{v_i}=-h(v_i)$, it follows that $\varphi\in (F^*)_a$ if and only if $\varphi(g_i)= \mu(i)t^{a+h(v_i)}$ with $\mu(i)\in K$ and $\mu(i)=0$ if $a+h(v_i)\not \in H$. Hence
\begin{eqnarray*}
(U^*)_a\iso \{\mu\in K^s\: \text{$\mu(i)=0$ for $i\in [s]\setminus \MF_a$ and $(\sum_{i\in \MF_a}\mu(i)t^{a+h(v_i)}g_i^*)(C)=0$}\}.
\end{eqnarray*}
In order to complete the proof of (\ref{equal}) we only need to prove the following statement:
  \begin{eqnarray}
  \label{state}
  (\sum_{i=1}^s\mu(i)t^{a+h(v_i)}g_i^*)(C)=0 \mbox{\ if and only if\  }  \langle \mu,\lambda\rangle=0 \mbox{\  for all  \  } \lambda\in W.
\end{eqnarray}

Let $z\in C_b$ for some $b\in  H$.  Then $z= \sum_{i\in [s]}\lambda(i) t^{b-h(v_i)}g_i$ with $\lambda(i)\in K$ for $i=1,\ldots,s$ and $\lambda(i)=0$ if $b-h(v_i)\notin H$ since $z\in \MF_a$. Moreover,  since $z\in \Ker(F\to U)$ it follows that $\lambda(1)t^{b-h(v_1)}df_{v_1}+\cdots+\lambda(s)t^{b-h(v_s)}df_{v_s}=0$. This implies that
 \[
\sum_{i\in [s]\atop b-h(v_i)\in H}\sum_{j\in [n]}\lambda(i)t^{b-h(v_i)}v_i(j)t^{h(v_i)-h_j}dx_j=\sum_{j\in [n]}(\sum_{i\in [s] \atop b-h(v_i)\in H}\lambda(i)v_i(j)t^{b-h_j})dx_j=0.
\]
Note that  if $b-h_j\notin H$, then for all $i\in [s]$ with $b-h(v_i)\in H$, one has $h(v_i)-h_j\notin H$ and so $v_i(j)=0$. Here we use the definition of $h(v_i)$, see (\ref{h(v)}). Therefore,  $\sum_{i\in [s], b-h(v_i)\in H}\lambda(i)v_i(j)=0$ for $j=1,\ldots,n$. This implies $\sum_{i\in [s], b-h(v_i)\in H}\lambda(i)v_i=0$.

In conclusion we see that
\begin{eqnarray*}
\label{iff}
\sum_{i\in [s], b-h(v_i)\in H}\lambda(i) t^{b-h(v_i)}g_i\in C_b \quad\text{ if and only if } \sum_{i\in [s], b-h(v_i)\in H}\lambda(i)v_i=0.
\end{eqnarray*}
This particularly  implies that if $z= \sum_{i\in [s]}\lambda(i) t^{b-h(v_i)}g_i\in C_b$, then $\lambda=(\lambda(1),\ldots,\lambda(s))\in W$.

Since
\[
(\sum_{i=1}^s\mu(i)t^{a+h(v_i)}g_i^*)(\sum_{i\in [s], b-h(v_i)\in H}\lambda(i) t^{b-h(v_i)}g_i)=
(\sum_{i=1}^s\mu(i)\lambda(i))t^{a+b},
\]
it follows that $(\sum_{i=1}^s\mu(i)t^{a+h(v_i)}g_i^*)(C_b)=0$ if and only if either $a+b\notin H$ or $\langle \mu,\lambda\rangle =\sum_{i=1}^s\mu(i)\lambda(i)=0$ for all $\lambda\in W$ satisfying  $\lambda(i)=0$ for all $i$ with $b-h(v_i)\notin H$. In particular, we have if $\langle \mu,\lambda\rangle=0 \mbox{\  for all  \  } \lambda\in W$, then $(\sum_{i=1}^s\mu(i)t^{a+h(v_i)}g_i^*)(C)=0$.

\medskip
 For the converse, we assume that $(\sum_{i=1}^s\mu(i)t^{a+h(v_i)}g_i^*)(C)=0$. Write $a=a_+-a_-$ with $a_+\in H$ and $a_-\in H$, and  set $b_0=\sum_{i=1}^sh(v_i)+a_-$.  Since $a+b_0\in H$ and
 $b_0-h(v_i)\in H$ for all $i\in [s]$, and since $(\sum_{i=1}^s\mu(i)t^{a+h(v_i)}g_i^*)(C_{b_0})=0$,  it follows that $(\mu,\lambda)=0$ for all $\lambda\in W$. Therefore the statement (\ref{state}) has been proved and this completes the proof.
  \end{proof}

\medskip
Now for any $a\in \ZZ H$ we want to determine the dimension of $(\Im \delta^*)_a$. We observe that the $\ZZ H$-graded $R$-module $\Im \delta^*$ is generated by the elements
\[
\delta^*((dx_i)^*)=\sum_{j=1}^s(\partial f_{v_j}/\partial x_i)g_j^*=
\sum_{j=1}^s v_j(i)t^{h(v_j)-h_i}g_j^*.
\]
Note that $\deg \delta^*((dx_i)^*)=-h_i$ for $i=1,\ldots ,n$.

For $i=1,\ldots,n$ we set  $w_i=(v_1(i),\ldots,v_s(i))$, and  for  $a\in \ZZ H$ we let  $KD_a$ be the $K$-subspace of $K^s$ spanned by the vectors $w_i$ for which $i\in \MG_a$. Here the set $\MG_a$ is defined to be
\[
\MG_a=\{i\in[n]\:\; a+h_i\in H\}.
\]

 \begin{Proposition}
 \label{F2}
 Let  $a\in \ZZ H$. Then  \[\dim_K (\Im \delta^*)_a=\dim_K KD_a.\]
 \end{Proposition}

 \begin{proof}  The $K$-subspace $(\Im \delta^*)_a\subset (F^*)_a$ is spanned by the vectors
\[
t^{a+h_i}\delta^*((dx_i)^*)=\sum_{j=1}^sv_j(i)t^{a+h(v_j)}g_j^*
\]
with $i\in \MG_a$.

The desired formula $\dim_K (\Im \delta^*)_a$ follows once we have shown that
\[\sum_{i\in \MG_a}\mu(i)t^{a+h_i}\delta^*((dx_i)^*)=0 \text{ if and only if } \sum_{i\in \MG_a}\mu(i)w_i=0.
\]
Here $\mu(i)\in K$ for any $i\in \MG_a$.
To prove it we notice that
\[
\sum_{i\in \MG_a}\mu(i)t^{a+h_i}\delta^*((dx_i)^*) =\sum_{i\in \MG_a}\mu(i)(\sum_{j=1}^sv_j(i)t^{a+h(v_j)}g_j^*)=
\sum_{j=1}^s(\sum_{i\in \MG_a}\mu(i)v_j(i)t^{a+h(v_j)}g_j^*).
\]
Thus $\sum_{i\in \MG_a}\mu(i)t^{a+h_i}\delta^*((dx_i)^*)=0$ if and only if  $\sum_{i\in \MG_a}\mu(i)v_j(i)=0$ for $j=1,\ldots,s$. Since $v_j(i)=w_i(j)$, this is the case if and only if $\sum_{i\in \MG_a}\mu(i)w_i=0$.
 \end{proof}

\begin{Corollary}
\label{rigid}
Let $a\in \ZZ H$. Then $\dim_K KD_a+\dim_K KL_a\leq \dim_K KL$.  Equality holds if and only if $T^1(R)_a=0$.
\end{Corollary}

Summarizing our discussions of this section we observe that all information which is needed to compute $\dim_KT^1(R)_a$ can be obtained  from the ($s\times n$)-matrix
\[
A_H= \left( \begin{array}{cccc}
v_1(1) & v_1(2) & \ldots & v_1(n) \\
v_2(1) & v_2(2) & \ldots & v_2(n)\\
\vdots  & \vdots  &  & \vdots \\
v_s(1) & v_s(2) & \ldots & v_s(n)
\end{array} \right).
\]
 Indeed, $\dim_K T^1(K[H])_a$  can be computed as follows: let $l=\rank A_H$, $r_a$ the rank of the submatrix of $A_H$ whose rows  are  the $i$th rows of $A_H$ for which $a+h(v_i)\not \in H$, and let $c_a$  be the rank of the submatrix of $A_H$ whose columns  are  the $j$th columns of $A_H$ for which $a+h_j\in H$. Then
\begin{eqnarray}
\label{finalformua}
\dim_K T^1(K[H])_a=l-l_a-d_a.
\end{eqnarray}

\medskip

\begin{Corollary}
\label{inH}
Suppose  $a\in H$. Then $T^1(R)_a=0$.
\end{Corollary}

\begin{proof}
Since $a\in H$, it follows that $\MG(a)=[n]$ and $\dim_KD_a=\dim_KKL=\rank A_H$. Thus the assertion follows from Corollary~\ref{rigid}.
\end{proof}

\medskip
The inequality of Corollary~\ref{rigid} can also be deduced from the following lemma.

\begin{Lemma} Fix $a\in \ZZ H$. Then $v_i(j)=0$ for every pair $i,j$ with $i\notin \mathcal{F}_a$ and $j\in \mathcal{G}_a$.
\end{Lemma}

\begin{proof} Assume on the contrary that $v_i(j)\neq 0$, say $v_i(j)<0$, for some  $i\notin \mathcal{F}_a$ and $j\in \mathcal{G}_a$. Then $$h(v_i)=-\sum_{k\atop v_i(k)<0}v_i(k)h_k=h_j+b,
\mbox{\ where\ } b=\sum_{k\neq j\atop v_i(k)<0}-v_i(k)h_k+(-v_i(j)-1)h_j\in H.$$ Since $j\in \mathcal{G}_a$, we have $a+h_j\in H$ and so $a+h(v_i)=(a+h_j)+b\in H$. Consequently,  $i\in \mathcal{F}_a$, a contradiction.
\end{proof}

\section{Separable and inseparable saturated lattices}

In this section we study conditions under which an affine semigroup ring $K[H]$ is obtained from another affine semigroup ring $K[H']$ by specialization, that is, by reduction modulo a regular element. Of course we can always choose $H'=H\times \NN$ in which case $K[H']$ is isomorphic to the polynomial ring $K[H][y]$ over $K[H]$ in the variable $y$, and  $K[H]$ is obtained from $K[H']$ by reduction modulo the regular element $y$. This trivial case we do not consider as a proper solution of finding an $K[H']$ that specializes to $K[H]$. If no non-trivial $K[H']$ exists, which specializes to $K[H]$, then $H$ will be called inseparable and otherwise separable. It turns out that the separability of $H$ is  naturally phrased in terms of the relation lattice $L$ of $H$.

\medskip
Let $L\subset \ZZ^n$ be a subgroup of $\ZZ^n$. Such a subgroup is  often called a lattice. The ideal $I_L$ generated by all binomials $f_v$ with $v\in L$ is called the {\em lattice ideal} of $L$. The following properties are known to be equivalent:
\begin{enumerate}
\item[(i)] $\ZZ^n/L$ is torsionfree;
\item[(ii)] $I_L$ is a prime ideal;
\item[(iii)] there exists a semigroup $H$ such that $I_L=I_H$.
\end{enumerate}
A proof of these facts can be found for example in \cite{EH}. A lattice $L$ for which $\ZZ^n/L$ is torsionfree is called a {\em saturated lattice}.

Let $\epsilon_1,\ldots,\epsilon_n$ be the canonical basis of $\ZZ^n$ and $\epsilon_1,\ldots,\epsilon_n,\epsilon_{n+1}$ the canonical basis of $\ZZ^{n+1}$. Let $i\in [n]$. We denote by $\pi_i\:\; \ZZ^{n+1}\to \ZZ^n$ the group homomorphism with $\pi_i(\epsilon_j)=\epsilon_j$ for $j=1,\ldots,n$ and $\pi_i(\epsilon_{n+1})=\epsilon_i$.  For convenience we denote again by $\pi_i$ the $K$-algebra homomorphism $S[x_{n+1}]\to S$ with $\pi_i(x_j)=x_j$ for $j=1,\ldots,n$ and $\pi_i(x_{n+1})=x_i$.

\begin{Definition}
\label{separable}
{\em Let $L\subset \ZZ^n$ be a saturated lattice. We say that   $L$ is {\em $i$-separable} for some $i\in[n]$, if there exists a saturated lattice $L'\subset \ZZ^{n+1}$ such that
\begin{enumerate}
\item[(i)] $\rank L'=\rank L$;
\item [(ii)] $\pi_i(I_{L'})=I_L$;
\item[(iii)] there exists a minimal system of generators   $f_{w_1},\ldots, f_{w_s}$ of $I_{L'}$ such that the vectors  $(w_1(n+1),\ldots,w_s(n+1))$ and $(w_1(i),\ldots,w_s(i))$ are linearly independent.
\end{enumerate}
The lattice $L$ is called {\em $i$-inseparable} if it is not $i$-separable, and $L$ is called {\em inseparable} if it is $i$-inseparable for all $i$. Moreover, the lattice $L'$ satisfying (i)-(iii) is called an {\em $i$-separation lattice} for $L$. We also call  a semigroup $H$ and its toric ring {\em inseparable} if the relation lattice of $H$ is inseparable.}
\end{Definition}

\begin{Remark}
\label{china}
{\em
Suppose that $L'\subset \ZZ^{n+1}$ is an $i$-separation lattice for $L$. Let $I_{L'}\subset S[x_{n+1}]$ be the lattice ideal  of $L'$. It is easily seen that  $x_{n+1}-x_i\not\in I_{L'}$ because  $\rank L=\rank L'$. Indeed, if  $x_{n+1}-x_i\in I_{L'}$, then $S[x_{n+1}]/I_{L'}\cong S/I_L$,  and so  $\rank L'=  \height I_{L'}=\height I_L+1=\rank L+1$, contradicting Definition~\ref{separable}(i).  Moreover, $x_{n+1}-x_i$ is a non-zerodivisor of $S[x_{n+1}]/I_{L'}$ since $S[x_{n+1}]/I_{L'}$ is a domain. In particular, if $f_{w_1},\ldots,f_{w_s}$ is a minimal system of generators of $I_{L'}$, then $\pi_i(f_{w_1}),\ldots, \pi_i(f_{w_s})$ is a minimal system of generators of $I_L$,  (see Lemma~\ref{minimal} for the details).  This implies that
\[
w_j(i)w_j(n+1)\geq 0 \quad \text{for}\quad  j=1,\ldots,s.
\]
Indeed, $x_i$ divides $\pi_i(f_{w_j})$ if $w_j(i)w_j(n+1)< 0$. Since   $\pi_i(f_{w_j})$ is a minimal generator of $I_L$ and since $I_L$ is a prime ideal, the polynomial $\pi_i(f_{w_j})$ must be irreducible. So,  $w_j(i)w_j(n+1)< 0$ is not possible.

  Let $v_j=\pi_i(w_j)$ for $j=1,\ldots, s$. Since $w_j(i)w_j(n+1)\geq 0$  for $j=1,\ldots,s$, for all $j$  we have $\pi_i(f_{w_j})=f_{v_j}$. Hence $f_{v_1}\,\ldots,f_{v_s}$ is a minimal system of generators of $I_L$.
}
\end{Remark}

For an affine semigroup $H\subset \ZZ^m$ the semigroup ring $K[H]$ is standard graded, if and only if there exists a  linear form
$\ell=a_1z_1+a_2z_2+\cdots +a_mz_m$ in the polynomial ring  $\QQ[z_1,\ldots,z_m]$ such that $\ell(h_i)=1$ for all minimal generators $h_i$ of $H$.

\medskip
The following  result provides a necessary condition of $i$-inseparability. Recall from \cite{BH} that an affine semigroup $H$ is called {\em positive} if $H_0=\{0\}$, where $H_0$ is the set of invertible elements of $H$.

\begin{Theorem}
\label{sufficient}
Let $H$ be a positive  affine semigroup which is minimally generated by $h_1,\ldots,h_n$,  $L\subset \ZZ^n$ the relation lattice of $H$.  Suppose that $L$ is $i$-separable.
Then
$
T^1(K[H])_{-h_i}\neq 0.
$
In particular,  if  $K[H]$ is standard graded, then $L$ is inseparable, if  $T^1(K[H])_{-1}=0.$
\end{Theorem}

\begin{proof} Since  $L$ is $i$-separable, there exists a saturated lattice  $L'$  satisfying the conditions (i) and (ii) as given  in Definition~\ref{separable}. Since $x_{n+1}-x_i$ is a non-zerodivisor on $R'=S[x_{n+1}]/I_{L'}$ it follows that $R''=R'/(x_{n+1}-x_i)^2R'$ is an infinitesimal deformation of $R$ (which is isomorphic to $R''/((x_{n+1}-x_i)R''$).

Let $v_j=\pi_i(w_j)$ for $j=1,\ldots,s$. By Remark~\ref{china}, we have $\pi_i(f_{w_j})=f_{v_j}$ for $j=1,\ldots,s$ and $f_{v_1},\ldots,f_{v_s}$ is a minimal system of generators of $I_L$.

Note that $S[x_{n+1}]=S[x_{n+1}-x_i]$. We set $\epsilon$ to be the residue class of $x_{n+1}-x_i$ in $S[x_{n+1}-x_i]/(x_{n+1}-x_i)^2$. Then $S[x_{n+1}-x_i]/(x_{n+1}-x_i)^2=S[\epsilon]$. Let $\sigma\: S[x_{n+1}]\to S[\epsilon]$ the canonical epimorphism and let $J$ be the image of $I_{L'}$ in $S[\epsilon]$. Then $R''=S[\epsilon]/J$.

 In order to determine the generators of $J$, we fix a $j$ with $1\leq j\leq s$, and may assume that $w(n+1)\geq 0$ and $w(i)\geq 0$. Then modulo $(x_{n+1}-x_i)^2$, we obtain
\[
f_{w_j}=\prod_{1\leq k\leq n \atop w_j(k)\geq 0}x_k^{w_j(k)}x_{n+1}^{w_j(n+1)}-\prod_{1\leq k\leq n\atop w_j(k)< 0}x_k^{w_j(k)}\qquad\qquad\qquad\qquad\quad
\]
\[
\qquad\quad=\prod_{1\leq k\leq n\atop w_j(k)\geq 0}x_k^{w_j(k)}(x_i^{w_j(n+1)}+w_j(n+1)x_i^{w_j(n+1)-1}\epsilon)-\prod_{1\leq k\leq n\atop w_j(k)< 0}x_k^{w_j(k)}
\]
\[
=f_{v_j}+[w_j(n+1)(\prod_{1\leq k\leq n\atop v_j(k)\geq 0}x_k^{v_j(k)})/x_i]\epsilon.\qquad\qquad\qquad\qquad\qquad\quad
\]
For the second equality we used that  $x_{n+1}=\epsilon+x_i$ and $\epsilon^2=0$, and the third equality  is due to the fact $v_j(i)=w_j(i)+w_j(n+1)$.

The homomorphism $\varphi: I_L/I_L^2\rightarrow R$ corresponding to the infinitesimal deformation  $S[\epsilon]/J$ is  given by
\[
\varphi(f_{v_j}+I_L^2)=w_j(n+1)(\prod_{1\leq k\leq n\atop v_j(k)\geq 0}x_k^{v_j(k)})/x_i+I_L=w_j(n+1)t^{h(v_j)-h_i} \mbox{ for }j=1,\ldots,s,
\]
which induces  the  element $$\alpha=\sum_{1\leq j\leq s}w_j(n+1)t^{h(v_j)-h_i}g_j^*\in (U^*)_{-h_i}.$$

Since $H$ is positive it follows that $\MG_{-h_i}=\{i\}$, and this  implies  that $(\mathrm{Im}\delta^*)_{-h_i}=K\sum_{1\leq j\leq s}v_j(i)t^{h(v_j)-h_i}g_j^*$, see Proposition~\ref{F2}. Assume $\alpha\in(\mathrm{Im}\delta^*)_{-h_i}$. Then there exists $\lambda\in K$ such that
\[
(w_1(n+1),\ldots,w_s(n+1)) =\lambda (v_1(i),\ldots,v_s(i)).
\]
Since $v_j(i)=w_j(i)+w_j(n+1)$ for $j=1,\ldots,s$, and since by condition (iii) of Definition~\ref{separable} the vectors
$(w_1(n+1),\ldots,w_s(n+1)$ and $(w_1(i),\ldots,w_s(i))$ are linearly independent, we obtain a contradiction. Hence  $T^1(R)_{-h_i}\neq 0$, as required.
\end{proof}

As a first example of a separable lattice we consider the relation lattice of a numerical semigroup.

\begin{Discussion}
\label{discussion}
{\em
Let $H\subset \NN$ be the  numerical semigroup minimally generated by $h_1,h_2,h_3$ with $\gcd(h_1,h_2,h_3)=1$. Recall some facts from \cite{H}. For  $i=1,2,3$ let $c_i$ be the smallest integer such that $c_ih_i\in \NN h_k+ \NN h_\ell$, where $\{i,k,\ell\}=[3]$, and let $r_{ik}$ and $r_{i\ell}$  be nonnegative integers such that  $c_ih_i=r_{ik}h_k+r_{i\ell}h_{\ell}$. Denote by $L$  the relation lattice of $H$. Then the  three vectors
\[
v_1=(c_1,-r_{12},-r_{13}),\quad  v_2=(-r_{21},c_2,-r_{23}),\quad  v_3=(-r_{31},-r_{32},c_3)
\]
generate  $L$. We have   $v_1+v_2+v_3=0$  if
\begin{enumerate}
\item[(1)] all $r_{ij}\neq 0$,  or
\item[(2)] $v_1=(c_1,-c_2,0)$, $v_2=(0,c_2,-c_3)$ and $v_3=(-c_1,0,c_3)$.
\end{enumerate}
In case (1), $f_{v_1},f_{v_2},f_{v_3}$ is the unique minimal system of generators of $I_L$. In case (2), $f_{v_1}+f_{v_2}+f_{v_3}=0$, so that any two of the $f_{v_i}$'s minimally generate $I_L$.

 An example for (1) is the semigroup with generators  $3$, $4$ and $5$, and example for (2) is the semigroup with generators  $6$, $10$ and $15$.
\medskip
(3) If $v_1+v_2+v_3\neq 0$, then there exist  distinct integers $k, \ell\in [3]$ such that  $v_k+v_{\ell}=0$ and $r_{ij}\neq 0$  for $i\in [3]\setminus \{k,\ell\}$ and $j\in \{k,\ell\}$. In this case $I_L$ is minimally generated by $x_i^{c_i}-x_{k}^{r_{ik}}x_l^{r_{il}}$ and $x_{k}^{c_k}-x_{\ell}^{c_{\ell}}$.

 An example for (3) is the semigroup with generators $4$, $5$ and $6$.}
\end{Discussion}

 It is known and easy to prove that $R=K[H]$ is not rigid. Indeed, since $R$ is quasi-homogeneous, the Euler relations $\sum_{i=1}^n(\partial f/\partial x_i)x_i=(\deg f)f$ imply that there is an epimorphism $\chi\: \Omega_{R/K}\to \mm$ with $\chi(dx_i)\mapsto t^{h_i}$  where $\mm=(t^{h_1}, t^{h_2},t^{h_3})$ is the graded maximal ideal of $R$. Since $\rank \Omega_{R/K}=\rank \mm =1$, it follows that $C=\Ker \chi$ is a torsion module. Thus we obtain the following exact sequence
 \[
 0\to C \to \Omega_{R/K}\to \mm\to 0,
 \]
 which induces the long exact sequence
 \[
 \Hom_R(C, R)\to   \Ext^1_R(\mm, R)\to \Ext^1_R(\Omega_{R/K}, R).
 \]
 Since $R$ is a $1$-dimensional domain, $R$ is Cohen-Macaulay,  $\Hom_R(C, R)=0$ and $\Ext^1_R(\mm, R)\iso \mm^{-1}/R\neq 0$. It follows that  $\Ext^1_R(\Omega_{R/K}, R)\neq 0$. In other words,  $R$ is not rigid.

 Of course the same argument can be applied to any numerical semigroup generated by more than 1 element.

 \medskip
 We have seen that $K[H]$ is not rigid. The next result shows that the relation lattice  of $H$ is even $i$-separable for  $i\in [3]$ with $T^1(R)_{-h_i}\neq 0$. To prove this we need

\begin{Lemma}\label{minimal}  Let $L\subset \ZZ^n$ and $L'\subset \ZZ^{n+1}$ be saturated lattices which satisfy the conditions {\em (i) and (ii)} as given in Definition~\ref{separable}. Then
\begin{enumerate}
\item[(a)] $I_L$ and $I_{L'}$ have the same number of minimal generators;

\item[(b)] $f_{w_1},\ldots,f_{w_s}$  is a minimal system of generators of $I_{L'}$ if and only if  \\ $\pi_i(f_{w_1}),\ldots,\pi_i(f_{w_s})$  is a minimal system of generators of $I_{L}$.
\end{enumerate}
\end{Lemma}

\begin{proof} (a) For any $S[x_{n+1}]/I_{L'}$-module $M$ we denote by $\overline{M}$ its reduction modulo $x_{n+1}-x_i$. The conditions (i) and (ii) of Definition~\ref{separable} guarantee that $x_{n+1}-x_i$ is a non-zerodivisor on $S[x_{n+1}]/I_{L'}$ and that $S/I_L\iso \overline{S[x_{n+1}]/I_{L'}}$. From these facts (a) follows.

(b) Suppose that $f_{w_1},\ldots,f_{w_s}$  is a minimal system of generators of $I_{L'}$. Then $I_L$ is generated by
$\pi_i(f_{w_1}),\ldots,\pi_i(f_{w_s})$ since $\pi_i(I_L')=I_L$.  By (a), $\pi_i(f_{w_1}),\ldots,\pi_i(f_{w_s})$ is a minimal system of generators of $I_L$.

Conversely, assume that  $\pi_i(f_{w_1}),\ldots,\pi_i(f_{w_s})$  is a minimal system of generators of $I_{L}$. We want to show that $I_{L'}=(f_{w_1},\ldots,f_{w_s})$. Set $J=(f_{w_1},\ldots,f_{w_s})$. Then we obtain the following short exact sequence:
\[0\rightarrow I_{L'}/J\rightarrow S[x_{n+1}]/J\stackrel{\alpha}{\rightarrow} S[x_{n+1}]/I_{L'}\rightarrow 0
\]
Here $\alpha$ is the natural epimorphism. By \cite[Proposition 1.1.4]{BH}, we obtain the  exact sequence
\[
0\rightarrow \overline{I_{L'}/J}\rightarrow \overline{S[x_{n+1}]/J}\stackrel{\overline{\alpha}}{\To} \overline{S[x_{n+1}]/I_{L'}}\rightarrow 0.
\]
Since $\pi_i(J)=\pi_i(I_L')=I_L$ it follows that  $\overline{\alpha}$ is an isomorphism,  and so  $\overline{I_{L'}/J}=0$. Nakayama's Lemma implies  that $I_{L'}/J=0$.  Hence $J=I_{L'}$, as desired.
\end{proof}

\begin{Proposition} Let  $H$ be a numerical semigroup as above  and set $R=K[H]$. Let $L\subset \ZZ^3$ be the relation lattice of $H$. With the notation of Discussion~\ref{discussion} we have:
\begin{enumerate}
\item[(a)] If $v_1+v_2+v_3=0$,   then $\dim_KT^1(R)_{-h_i}=1$ and $L$ is $i$-separable for $i=1,2,3$.
\item[(b)] If $v_1+v_2+v_3\neq 0$, then there exists $i\in [3]$ such that  $I_L=(x_i^{c_i}-x_k^{r_{ik}}x_l^{r_{il}}, x_k^{c_\ell}-x_{\ell}^{c_k})$ with $\{i,k,l\}=[3]$ and $r_{ik}, r_{il}\neq 0$.  In this case,  $T^1(R)_{-h_i}= 0$,  and for $j\neq i$ we have that $T^1(R)_{-h_j}\neq 0$   and that $L$ is $j$-separable. \end{enumerate}
\end{Proposition}
\begin{proof} (a) We consider the  case (1),   where $r_{ij}>0$  for all $i$ and $j$, see Discussion~\ref{discussion}. Fix $i\in [3]$.  Since all $r_{ij}>0$ it follows  that $\mathcal{F}_{-h_i}=\{1,2,3\}$,  and since $H$ is a positive semigroup we have  $\MG_{-h_i}=\{i\}$. It follows from Corollary~\ref{F1} and Proposition~\ref{F2} that $\dim_K (U^*)_{-h_i}=2$ and $\dim_K (\mathrm{Im} (\delta^*_{-h_i})=1$. Hence $\dim_KT^1(R)_{-h_i}=1$.

Consider the vectors
\begin{eqnarray*}
w_1&=&(c_1-1,-r_{12},-r_{13},1),\\
 w_2&=&(-r_{21}+1,c_2,-r_{23},-1),\\
 w_3&=&(-r_{31},-r_{32},c_3,0)
\end{eqnarray*}
in $\ZZ^4$, and set $L'=\ZZ w_1+\ZZ w_2+\ZZ w_3$. We will prove  that $L'$ is a $1$-separation of $L$. First we show that $L'$ is saturated. Indeed, if $aw\in L'$ for some $0\neq a\in \ZZ$ and some $w\in \ZZ^4$, then $aw=a_1w_1+a_2w_2+a_3w_s$ for some $a_i\in \ZZ$,  and it follows that $av=a_1v_1+a_2v_2+a_3v_3$, where $v=\pi_1(w)$. This  implies that $v=k_1v_1+k_2v_2+k_3v_3$  for some $k_i\in\ZZ$,  since $L$ is saturated.  Thus $(a_1-ak_1)v_1+(a_2-ak_2)v_2+(a_3-ak_3)v_3=0$ and so $a_1-ak_1=a_2-ak_2=a_3-ak_3$. It follows that $(a_1-ak_1)w_1+(a_2-ak_2)w_2+(a_3-ak_3)w_3=0$. Thus $w=k_1w_1+k_2w_2+k_3w_3$. Hence $L'$ is saturated. Next we show $\pi_1(I_{L'})=I_L$.
It is clear that $I_L\subseteq\pi_1(I_{L'})$ since $\pi_1(f_{w_i})=f_{v_i}$ for $i=1,2,3$. For the converse direction, we only need to note that $\pi_1(L')=L$ and that $f_{\pi_1(w)}$ divides $\pi_1(f_w)$ for all $w\in L'$.

Now, applying  Lemma~\ref{minimal} we conclude that $f_{w_1},f_{w_2},f_{w_3}$ is a minimal system of generators of $I_{L'}$ satisfying  condition  (iii) of Definition~\ref{separable}. Consequently, $L'$  is a $1$-separation of $L$. Similar arguments work for $i=2,3$.

\medskip
In case (2), $L$ is generated by any two of the vectors $v_1=(c_1,-c_2,0)$,  $v_2=(0,c_2,-c_3)$ and $v_3=(-c_1,0,c_3)$. Let $L'\subset \ZZ^4$ be a lattice generated by $w_1=(c_1-1,c_2,0,1)$ and $w_2=(-c_1,0,c_3,0)$. We claim that $L'$ is a $1$-separation of $L$.  Indeed, the ideal of $2$-minors $I_2(W)$ of the matrix $W$ whose row vectors are $w_1$ and $w_2$ contains the elements $c_1$ and $c_3$. By the choice of the $c_i$'s it follows that $\gcd(c_1,c_3)=1$. Thus, $I_2(W)=\ZZ$. This shows that $L'$ is saturated. Since $\pi_1(f_{w_i})=f_{v_i}$ for $i=1,2$ and since $I_L=(f_{v_1},f_{v_2})$, Lemma~\ref{minimal} implies that $I_{L'}=(f_{w_1},f_{w_2})$ and $\pi_1(I_{L'})=I_L$. Since $L'$ satisfies also condition (iii) of Definition~\ref{separable}, it follows that $L$ is  $1$-separable. In the same way it is shown that $L$ is $i$-separable for $i=2,3$.

(b) This is case (3) of Discussion~\ref{discussion} and  we have $I_H=(x_i^{c_i}-x_{k}^{r_{ik}}x_l^{r_{kl}}, x_{k}^{c_k}-x_{\ell}^{c_{\ell}})$ with $\{i,k,l\}=[3]$. Thus $I_H$ is a complete intersection and the exponents $c_1,c_2$ and $c_3$ are all $>1$. Without loss of generality  we may assume that $i=2$, $k=1$ and $l=3$. Since the lattice $L\subset \ZZ^3$ with  basis  $v_1=(-r_{21},c_2,-r_{23})$, $v_2=(-c_3,0,c_1)$ is saturated, it follows that  the ideal of $2$-minors $(c_1c_2,c_2c_3, c_1r_{21}+c_3r_{23})$ of
\[
\left( \begin{array}{ccc}
-r_{21} & c_2 & -r_{23}\\
-c_3 & 0 & c_1
\end{array} \right)
\]
is equal to $\ZZ$.

Consider the lattice $L'\subset \ZZ^4$ whose basis $w_1,w_2$ consists of  the row vectors of
\[
\left( \begin{array}{cccc}
-r_{21}+1 & c_2 & r_{23} & -1\\
-c_3 & 0 & c_1 & 0
\end{array} \right).
\]
The ideal of $2$-minors of this matrix contains $(c_1c_2,c_2c_3, c_1r_{21}+c_3r_{23})$, and hence is again equal to $\ZZ$. Thus $L'$ is saturated. Furthermore we have  $\pi_2(f_{w_1})=f_{v_1}$, $\pi_2(f_{w_2})=f_{v_2}$ and $\pi_2(L')=L$. This implies that $\pi_2(I_{L'})=I_L$. Since $\rank L'=\rank L=2$, the conditions (i) and (ii) of Definition~\ref{separable} are satisfied. Applying Lemma~\ref{minimal} we obtain $I_{L'}=(f_{w_1},f_{w_2})$. Since  the condition (iii) of Definition~\ref{separable} is also satisfied we see that $L$ is $1$-separable. Similarly, one shows that $L$ is $3$-separable.
\end{proof}

\section{Inseparable bipartite graphs}  Let $G$ be a finite simple graph on the vertex set $[m]$, and let  $K$ a field. The $K$-algebra  $R=K[G]=K[t_it_j\:\; \{i,j\}\in E(G)]$  is called   the {\em edge ring} of $G$. Here $E(G)$ denotes the set of edges of $G$. We let $n=|E(G)|$, and denote by $S$ the polynomial ring  over $K$ in the indeterminates $x_e$ with  $e\in E(G)$. Let $\varphi\: S\to K[G]$ be the $K$-algebra homomorphism with $x_e\mapsto t_it_j$ for $e=\{i,j\}$. The toric ideal $\Ker \varphi$ will be denoted by $I_G$.

In this section we will discuss   inseparability  of the edge ring of a bipartite graph,  which may as well be considered as the toric ring associated with the affine semigroup $H$ generated by the elements $\delta_i+\delta_j$ with $\{i,j\}\in E(G)$, where $\delta_1,\ldots,\delta_m$ is the canonical basis of $\ZZ^m$.

\medskip
  It is known (see e.g. \cite{HHBook}) that the generators of $I_G$ are given in terms of even cycles of $G$. Recall that a walk in $G$ is a sequence $C\: i_0,i_1,\ldots, i_q$ such that $\{i_k,i_{k+1}\}$ is an edge of $G$ for $k=0,1,\ldots,k-1$. $C$ is called a closed walk, if $i_q=i_0$. The closed walk $C$ is called a cycle if $i_j\neq i_k$ for all $j\neq k$ with $j,k<q$, and it is called an even closed walk if $q$ is even. Observe that any cycle of bipartite graph is an even cycle.

Given any even cycle (more generally an even closed walk) $C\: i_0,i_1,\ldots, i_{2q}$. The edges of $C$ are  $e_{k_j}=\{i_j,i_ {j+1}\}$ for  $j=0,1,\ldots,2q-1$ together with the  edge  $e_{k_{2q-1}}=\{i_{2q-1},i_0\}$.
 We associate to $C$  the  vector $v(C)\in\ZZ^n$   which defined as
 \begin{eqnarray}
 \label{vector}
 v(C)=\sum_{i=0}^{q-1}\epsilon_{k_{2i}}-\sum_{i=0}^{q-1}\epsilon_{k_{2i+1}}
 \end{eqnarray}
Here $\epsilon_1,\ldots, \epsilon_n$ denotes  the canonical basis of $\ZZ^n$. Note that $v(C)$ is determined by $C$ only up to sign. We call $v(C)$ as well as $-v(C)$  the vector corresponding to $C$.

For simplicity we write $f_C$ for $f_{v(C)}$.
 Recall from \cite{OH} that the toric ideal $I_G$ of a finite bipartite graph is minimally generated by {\it indispensable binomials}, that is, by binomials, which up to  sign,  belong to any system of generators of $I_G$. Furthermore, a binomial $f\in I_G$ is indispensable  if and only if $f=f_C$, where  $C$ is an induced  cycle, that is,  a cycle without a chord. In particular, if  $G'$ is  the graph obtained from $G$ by deleting  all edges which do not belong to any cycle,  then $I_G=I_{G'}S$.

\medskip
Now for the rest of this section we let $G$ be a bipartite  graph on the  vertex set $[m]$  with edge set $E(G)=\{e_1,\ldots,e_n\}$.  With  the  edge $e_k=\{i,j\}$ we associate the  vector $h_k=(h(e_k))=\delta_{i}+\delta_{j}$. Here $\delta_1,\ldots, \delta_m$ is the canonical basis of $\ZZ^m$. The semigroup generated by $h_1,\ldots,h_n$ we denote by $H(G)$ or simply by $H$. Note that $K[H(G)]=K[G]$.

Let $\{C_1,\ldots, C_s\}$   be the  set of cycles of $G$ and $v_i=v(C_i)$   the vector corresponding to $C_i$.  We may assume that for $ i=1,\ldots,s_1\leq s$, the cycles  $C_i$  are all the  induced cycles of $G$. Then $I_G$ is minimally  generated by $f_{v_1},\ldots,f_{v_{s_1}}$,  see \cite{OH}.  Of course,  $I_G$ is also generated by $f_{v_1},\ldots,f_{v_s}$. In particular, if $L$ is the relation lattice of $H$, then $KL$ is the vector space spanned by $v_1,\ldots,v_s$.

\medskip
 Let $a\in \ZZ H$.  As in Section~1 we set
\[
\mathcal{F}_a=\{1\leq i\leq s\:\; a+h(v_i)\in H\}, \mbox{ and  } KL_a=\mathrm{Span}_K\{v_i\:\; i\in [s]\setminus \mathcal{F}_a\}.
 \]
In addition we now also  set
 \[
 \MF'_a=\{1\leq i\leq s_1\:\; a+h(v_i)\in H\}, \mbox{ and  } KL'_a=\mathrm{Span}_K\{v_i\:\; i\in [s_1]\setminus \MF'_a\}.
 \]
In general, $\MF'_a$ is a proper subset of $\MF_a$. However, we have

\begin{Lemma}
\label{basic}
$KL'_a=KL_a$ for all  $a\in \ZZ  H$.
\end{Lemma}
\begin{proof} Since $[s_1]\setminus \MF'_a\subseteq [s]\setminus \MF_a$, we have $KL'_a\subseteq KL_a$. Let $i\in ([s]\setminus \MF_a)\setminus ([s_1]\setminus \MF'_a)$. Then $a+h(v_i)\notin H$ and $C_i$ is a cycle with chords. In the following we describe a process to obtain the induced cycles with vertex set contained in $V(C_i)$. Choose  a chord of $C_i$ and note that this chord divides $C_i$ into two cycles. If both cycles are induced, then the process stops. Otherwise we divide as before, those cycles which are not induced.  Proceeding in this way, we obtain  induced cycles of $G$, denoted by $C_{i_1},\ldots, C_{i_k}$,  such that $E(C_{i_j})$ consists of at least one  chord of $C_i$. Moreover,  the edges of $C_{i_j}$  which are not  chords of $C_i$,  are edges of $C_i$.

 In general, if $C$ is a cycle and $v=v(C)$, then
\[
h(v)=\sum_{j\in V(C)}\delta_j.
\] Hence it follows from the construction of the induced cycles $C_{i_j}$ that $h(v_i)-h(v_{i_j})$ is the  sum of certain terms $\delta_{k_1}+\delta_{k_2}$, where $\{k_1,k_2\}$ is an edge of $C_i$,  and hence $h(v_i)-h(v_{i_j})\in H$. Since  $a+h(v_i)\notin H$ it follows  that  $a+h(v_{i_j})\notin H$ for all $j$. This implies that $v_{i_j}\in KL'_a$ for all $j$, and so $v_i\in KL'_a$ since $v_i$ is a linear combination of the $v_{i_j}$.
\end{proof}

 For the discussion on separability we need to know when  $T^1(K[G])_{-h_j}$ vanishes, see Theorem~\ref{sufficient}. For that we need to have the  interpretation of $\MF_{-h_j}$ for edge rings which is given by the following formula:
\begin{eqnarray}
\label{minus}
\MF_{-h_j}=\{i\in [n] : V(e_j)\subset V(C_i)\}.
\end{eqnarray}
For the proof of this equation note that if $V(e_j)\subseteq V(C_i)$, then without loss of generality we assume that $C_i: 1,2,\ldots,2t$ and that $e_j=\{1,k\}$ with $k\in [2t]$. Note that $k$ is even, since $G$ contains no odd cycle. It follows that $-h_j+h(v_i)=(\delta_2+\delta_3)+\cdots+(\delta_{k-2}+\delta_{k-1})+(\delta_{k+1}+\delta_{k+2})+\cdots+(\delta_{2t-1}+\delta_{2t})\in H$, and so $i\in \MF_{-h_j}$ by definition. Conversely, assume that $V(e_j)\nsubseteq V(C_i)$ and let $k\in V(e_j)\setminus  V(C_i)$. Then $-h_j+h(v_i)$ is a vector in $\ZZ^m$ with the $k$th entry negative and thus it does not belong to $H$. Therefore $i\notin \MF_{-h_j}$.

\medskip
Later we  also shall need

\begin{Lemma}
\label{walk}
Let $W\: i_1,i_2,\ldots,i_{2k},i_1$ be an even closed   walk in $G$ and let $e_j$ be an edge of $G$ with the property that
$e_j\neq \{i_a,i_b\}$ with $1\leq a<b \leq 2k$. Then the vector $w=v(W)\in KL$  belongs to $KL_{-h_j}$.
\end{Lemma}
\begin{proof} We may view  $W$ as a bipartite graph with bipartition $\{i_1,i_3, \ldots, i_{2k-1}\}$ and $\{i_2,i_4, \ldots, i_{2k}\}$. Then we see that $w$ belongs to the space spanned by the vectors  corresponding to the  induced cycles of $G$ with edges in $W$. This  vector space is a subspace of $KL_{-h_j}$,  since $e_j$ is not an  edge of any  cycle with edges in $W$, as follows from (\ref{minus}).
\end{proof}

 We call the space $KL$ which is spanned by the vectors $v_1,\ldots,v_s$ the  {\em cycle space} of $G$ (with respect to $K$). Usually the cycle space is only defined over $\ZZ_2$. For bipartite  graphs the dimension of the cycle space does not depend on $K$ and is known to be
 \begin{eqnarray}
 \label{dimension}
 |E(G)|-|V(G)|+c(G).
 \end{eqnarray}
where $c(G)$ is the number of connected components of  $G$, see \cite[Corollary 8.2.13]{V}.

\medskip
\noindent {\em  Inseparability.}  In this subsection we will show that $K[G]$ is $j$-separable if and only if $T^1(K[G])_{-h_j}\neq 0$ for $j\in [n]$ and present a  characterizations of  bipartite graphs $G$ for which $K[G]$ is inseparable.

Note that (\ref{minus}) says that $i\in \MF_{-h_j}$ if and only if $e_j$ is an edge or a chord of $C_i$. Accordingly, we split the set $\MF_{-h_j}$ into the two subsets
\begin{eqnarray}
\label{aj}
\MA_j=\{i\in [s]\:\; e_j \text{ is an edge of } C_i\},
\end{eqnarray}
and
\begin{eqnarray}
\label{bj}
\MB_j=\{i\in [s]\:\; e_j \text{ is a chord of  } C_i\}.
\end{eqnarray}

\medskip
We also  set $V_{-h_j}=\Span_K \{v_i:i\in [s]\setminus \MA_j\}$.  Then, since  by assumption all edges of $G$ belong to a cycle, we obtain
\begin{eqnarray}
\label{begin}
\dim_K V_{-h_j}=\dim_K KL-1\quad \text{for}  \quad j=1,\ldots,n.
\end{eqnarray}
Indeed, let $G\setminus \{e_j\}$ be the graph obtained from $G$ by deleting the edge $e_j$ and leaving vertices unchanged.
Then  $V_{-h_j}$ is the cycle space of $G\setminus \{e_j\}$.

\begin{Lemma} $T^1(R)_{-h_j}=0$ if and only if for all $i\in \MB_j$, one has $v_i\in KL_{-h_j}.$
\label{lemma3.2}
\end{Lemma}
\begin{proof} Since $-h_j+h_i\in H$ if and only if $i=j$ it follows that $\dim (\Im \delta^*)_{-h_j}=1$, see Proposition~\ref{F2}. Thus,  since $KL_{-h_j}\subseteq V_{-h_j}$, it follows from (\ref{begin})  that $T^1(R)_{-h_j}=0$ if and only if $V_{-h_j}=KL_{-h_j}$. Since $V_{-h_j}=KL_{-h_j}+\Span_K \{v_i: i\in  \MB_j\}$, the assertion follows. \end{proof}

For stating the next result we have first to introduce some  concepts.  Let $C$ be a cycle. Then the path  $P\: i_1,i_2,i_3,\ldots,  i_{r-1},i_r$ (with $r\geq 2$ and with $i_j\neq i_k$ for all $j\neq k$) is called a {\em path chord} of $C$ if $i_1, i_r\in V(C)$ and $i_j\not\in V(C)$ for all $j\neq i_1, i_r$. The vertices $i_1$ and $i_r$  are  called the ends of $P$. Note that any chord of $C$ is a path chord.

Let $P$ be a path chord of $C$. We may assume that $\{i, i+1\}$ for $i=1,\dots, t$ together with $ \{1,2t\}$  are the edges of $C$ and that $i_1=1$ and $i_r=k$ with $k\neq 1$. Let $P'$ be another path chord of $C$. Then we say that $P$ and $P'$ {\em cross} each other if  one end of $P'$ belongs to the  interval $[2,k-1]$ and the other end of $P'$ belongs to $[k+1,2t]$.
In particular, if $P$ is a chord and $P'$ crosses $P$, we say that $P'$ is a  {\em crossing path chord} of $C$ with respect to the chord $P$.

\begin{Theorem}
\label{crossing}
Let $G$ be a bipartite graph with edge set  $\{e_1,\ldots,e_n\}$, and let $R=K[G]$ be the edge ring  of $G$. Then the  following conditions are equivalent:
\begin{enumerate}
\item[(a)] $T^1(R)_{-h_j}\neq 0$.
\item[(b)] There exists a cycle $C$ of $G$ for which $e_j$ is a chord, and there is no crossing path chord $P$ of $C$ with respect to $e_j$.
 \item[(c)] The relation lattice of $H(G)$ is $j$-separable.
 \end{enumerate}
\end{Theorem}

\begin{proof}
(a)\implies (b): Assume that (b) does not hold. Let $i\in \MB_j$ with  $\MB_j$ as defined in (\ref{bj}). By our assumption, $C_i$ admits  a  path chord, denoted by $P$, which crosses $e_j$. Denote by $i_1,i_2$ the two ends of $P$. Then $C$ is the union of two paths $P_1$ and $P_2$ which both have ends $i_1,i_2$. Since  $P_1\cup P$ and $P_2\cup P$ are  cycles and
  $e_j$ is neither an edge nor a chord of them, it follows from Lemma~\ref{walk} that  the  vectors  $w_1=v(P_1\union P)$ and $w_2=v(P_2\union P)$ belong to $KL_{-h_j}$.
Therefore, $v_i\in KL_{-h_j}$ because it  is a linear combination of $w_1$ and $w_2$. Now applying Lemma~\ref{lemma3.2}, we obtain $T^1(R)_{-h_j}=0$, a contradiction.

(b)\implies (c):  We may assume  that the cycle $C$ as given in (b) has the edge set
\[ E(C)=\{e_1=\{1,2\},\ldots, e_{\ell}=\{\ell,\ell+1\}, \ldots, e_{2t}=\{2t,1\}\},
\]
and that  $e_j=\{1,k\}$ with $2< k<2t-1$.

We let $X$ be the set of all $a\in [m]\setminus V(C)$ for which there is a path $P$  from  $a$
to some vertex of  $[2,k-1]$, and we set $Y=[m]\setminus (V(C)\cup X)$.

\medskip
We now define a graph $G'=G_1\union G_{2}$, where $G_1$ and $G_2$ are disjoint graphs, that is, $V(G_1)\sect V(G_2)=\emptyset$. The graph $G_2$ is the subgraph of $G$ induced on $X\union [k]$. Next we first define $\widetilde{G}_1$ as the subgraph of $G$ induced on $Y\union [k+1,2t]\union \{1,k\}$. Then $G_1$  is obtained from $\widetilde{G}_1$ by renaming $1$ as $m+1$ and $k$ as $m+2$. We claim  that $G_1$ and $G_2$ are disjoint. Indeed,  $V(G_1)\sect V(G_2)\subseteq [k+1,2t]\sect X$. Condition (b) implies that $[k+1,2t]\sect X=\emptyset$.

Now we claim that if we identify in $G'$ the vertex  $m+1$ with $1$ and the vertex $m+2$ with $k$, then we obtain $G$. Indeed, let $G''$ be the graph which is obtained from $G'$ after this identification. We have to show that $G''=G$. Obviously, we have $V(G'')=V(G)$ and $E(G'')\subseteq E(G)$. Let $e\in E(G)\setminus E(G'')$. Then $e=\{k_1,k_2\}$ with $k_1\in [k+1,2t]\union Y$ and $k_2\in X\union [2,k-1]$. If  $k_2\in [2,k-1]$, then $k_1\in [1,k]\cap X$ by the definition of $X$. This is impossible since $(X\union[1,k])\sect ([k+1,2t]\union Y)=\emptyset$;   If $k_2\in X$, then again by the definition of $X$ it follows that $k_1\in X\union[1,k]$,  which is impossible again in the same reason. Thus we have proved the claim.

  Now the edge ring of $G'$ is of the form $R'=S'/I_{G'}=S'/(I_{G_1}+I_{G_2})S'$, where $S'=S[x_{n+1}]$ and where  the variable $x_{n+1}$ corresponds to the edge $e_{n+1}=\{m+1,m+2\}$. The variable $x_j$ corresponds to the edge $e_j$ if  $e_j\in G$, and to $e\in E(G_1)$ if $e\in E(G_1)\setminus \{m+1,m+2\}$ and $e$ is mapped to $e_j$ by the identification map $G'\to G''=G$.  Let $L$ be the relation lattice of $ H(G)$ and $L'$ be the relation lattice of $H(G')$. Then $L\subset \ZZ^n$ and $L'\subset \ZZ^{n+1}$ are saturated lattices. We claim that $L$ and $L'$ satisfy the conditions (i), (ii) and (iii)  with respect to $\pi_j$, see  Definition~\ref{separable}.
We first show that $\pi_j(I_{L'})=I_L$. Let $f$ be a minimal generator of $I_L$. Then there exists an induced cycle $D$ of $G$ such that $f=f_D$. Since $G=G''$ it follows that $V(D)\subset V(\widetilde{G}_1)$ or $V(D)\subset V(G_2)$. Hence  there is an induced cycle $D'$ in $G'$ whose image under the identification map is $D$. Therefore, $\pi_j(f_{D'})=f_D$. This proves the condition (ii). Since (ii) is satisfied, it follows that $R'/(x_{n+1}-x_j)R'\iso R$. Moreover, $x_{n+1}-x_j$ is a non-zerodivisor on $R'$, since $R'$ is a domain. This implies that $\height I_{L'}=\height I_L$. In particular, $\rank L'=\rank L$. Thus  the condition (i) is also satisfied. Finally, by the definition of $G_1$ and $G_2$, there exist an induced cycle of $G_1$ with $e_{n+1}$ as an edge, say $C_1$,
and an induced cycle of $G_2$ with $e_j$ as an edge, say $C_2$.  Let $w_1=v(C_1)$ and $w_2=v(C_2)$. Then $w_1(n+1)\neq 0$, $w_1(j)=0$, $w_2(n+1)=0$ and $w_2(j)\neq 0$. This implies the condition (iii).

The implication (c)\implies (a) follows from Theorem~\ref{sufficient}.
\end{proof}

\begin{Definition}

\label{3.5}
{\rm We say that  a bipartite graph $G$  is {\it  separable by an edge} $e$, if there exist nonempty subgraphs $G_1$ and $G_2$  of $G$ such that

 (1) $E(G_1)\cap E(G_2)=\{e\}$ and $V(G_1)\cap V(G_2)=e$,

  (2) $E(G)=E(G_1)\cup E(G_2)$, $V(G)=V(G_1)\cup V(G_2)$,

  (3) the edge   $e$  is an {\it internal edge} of both $G_1$ and  $G_2$. Here an  edge of a graph is called internal,  if it belongs to a cycle of this graph.

A  bipartite graph is called {\it inseparable}, if it is not separable.}

  \end{Definition}

According to this definition, a  bipartite graph $G$ is separable if it can be obtained by gluing two disjoint bipartite graph along an internal edge. Figure 1 displays  a bipartite graph which  is separable by $e_1$ as well as by $e_2$.

  \begin{figure}[hbt]
\begin{center}
\psset{unit=1cm}
\begin{pspicture}(2.75,1.5)(9,4)

\rput(6,2.25){$\bullet$}
\rput(7.5,2.25){$\bullet$}
\rput(6,3.75){$\bullet$}
\rput(7.5,3.75){$\bullet$}

\psline[linewidth=0.6pt,linecolor=black](6,2.25)(7.5,2.25)
\psline[linewidth=0.6pt,linecolor=gray](7.5,2.25)(7.5,3.75)
\psline[linewidth=0.6pt,linecolor=gray](6,2.25)(6,3.75)

\psline[linewidth=0.6pt,linecolor=black](6,3.75)(7.5,3.75)

\rput(4.5,2.25){$\bullet$}
\psline[linewidth=0.6pt,linecolor=gray](4.5,2.25)(4.5,3.75)

\rput(4.5,3.75){$\bullet$}
\psline[linewidth=0.6pt,linecolor=black](4.5,2.25)(6,2.25)
\psline[linewidth=0.6pt,linecolor=black](4.5,3.75)(6,3.75)

\rput(3,2.25){$\bullet$}
\psline[linewidth=0.6pt,linecolor=gray](3,2.25)(3,3.75)

\rput(3,3.75){$\bullet$}
\psline[linewidth=0.6pt,linecolor=black](4.5,2.25)(3,2.25)
\psline[linewidth=0.6pt,linecolor=black](4.5,3.75)(3,3.75)

\rput(4.3,3){$e_1$}\rput(6.2,3){$e_2$}
\end{pspicture}
\end{center}
\caption{}\label{P}
\end{figure}

 With this concept introduced,  we can  show that a bipartite graph is inseparable in the sense of Definition~\ref{3.5} if and only if  $K[G]$ is inseparable. Thus the algebraic inseparability of $G$ has a combinatorial interpretation.

\begin{Corollary}
\label{unique}
Let $G$ be a bipartite graph. Then $K[G]$ is inseparable  if and only  $G$ is  inseparable.
\end{Corollary}

\begin{proof}
Assume first that $G$ is separable by an edge $e$. Then $e$ is a chord of a cycle say $C$ of $G$ by the condition (3) in Definition~\ref{3.5} and there is no crossing path chord of $C$ with respect to $e$. Hence $K[G]$ is separable by Theorem~\ref{crossing}.

Assume now that $K[G]$ is separable. In view of the proof (b)\implies (c) of  Theorem~\ref{crossing}, $G$ is obtained by gluing $G_1$ and $G_2$ along an internal edge. It follows that $G$ is separable, as required.
\end{proof}

As an example of the theory which we developed so far we consider coordinate rings of convex polyominoes. First we recall from \cite{Q} the definitions and some facts about  convex polyominoes.

Let $\mathbb{R}^2_+=\{(x, y) \in \mathbb{R}: x, y \geq 0\}$. We consider $(\mathbb{R}_+,\leq)$ as a partially ordered
set with $(x, y) \leq (z,w)$ if $x \leq z$ and $y \leq w$. Let $a, b \in \mathbb{Z}^2_+$. Then the set $[a, b] = \{c\in \mathbb{Z}^2_+: a \leq c \leq b\}$ is called an {\em interval}.

A {\em cell} $C$ is an interval of the form $[a, b]$, where $b = a + (1, 1)$. The elements of $C$
are called {\em vertices} of $C$. We denote the set of vertices of $C$ by $V(C)$. The intervals
$[a, a+(1, 0)], [a+(1, 0), a+(1, 1)], [a+(0, 1), a+(1, 1)]$ and $[a, a+(0, 1)]$ are called
{\em edges} of $C$. The set of edges of $C$  is denoted by $E(C)$.

Let $P$ be a finite collection of cells of $\mathbb{Z}^2_+$. Then two cells $C$ and $D$ are called
{\em connected} if there exists a sequence $ \mathcal{C}: C = C_1,C_2, \ldots ,C_t = D$ of cells of $\MP$ such
that for all $i = 1,\ldots , t - 1$ the cells $C_i$ and $C_{i+1}$ intersect in an edge. If the cells in
$\mathcal{C}$ are pairwise distinct, then $\mathcal{C}$ is called a {\em path} between $C$ and $D$. A finite collection
of cells $\MP$ is called a {\em polyomino} if every two cells of $\MP$ are connected. The vertex set
of $\MP$, denoted $V(\MP)$, is defined to be $\Union_{C\in \MP}V(C)$ and the edge set of $\MP$, denoted  $E(\MP)$,  is defined to be $\Union_{C\in \MP}E(C)$.  A polyomino is said to be {\em vertically} or {\em column convex} if its intersection with any vertical line is convex. Similarly, a polyomino is said to be {\em horizontally} or {\em row convex} if its intersection with any horizontal line is convex. A polyomino is said to be {\em convex} if it is row and column convex. Figure~\ref{P} shows two polyominos  whose cells are marked by gray color The right hand side polyomino is convex while the left one is not.

\begin{figure}[ht!]
\begin{tikzpicture}[line cap=round,line join=round,>=triangle 45,x=0.8cm,y=0.8cm]
\clip(4.76,0.7) rectangle (9.22,5.66);
\fill[fill=black,fill opacity=0.1] (5,1) -- (6,1) -- (6,2) -- (5,2) -- cycle;
\fill[fill=black,fill opacity=0.1] (7,1) -- (7,2) -- (6,2) -- (6,1) -- cycle;
\fill[fill=black,fill opacity=0.1] (8,1) -- (8,2) -- (7,2) -- (7,1) -- cycle;
\fill[fill=black,fill opacity=0.1] (6,3) -- (7,3) -- (7,4) -- (6,4) -- cycle;
\fill[fill=black,fill opacity=0.1] (6,3) -- (6,2) -- (7,2) -- (7,3) -- cycle;
\fill[fill=black,fill opacity=0.1] (7,4) -- (7,3) -- (8,3) -- (8,4) -- cycle;
\fill[fill=black,fill opacity=0.1] (6,4) -- (7,4) -- (7,5) -- (6,5) -- cycle;
\fill[fill=black,fill opacity=0.1] (8,5) -- (7,5) -- (7,4) -- (8,4) -- cycle;
\fill[fill=black,fill opacity=0.1] (8,4) -- (9,4) -- (9,5) -- (8,5) -- cycle;
\draw (5,1)-- (6,1);
\draw (6,1)-- (6,2);
\draw (6,2)-- (5,2);
\draw (5,2)-- (5,1);
\draw (7,1)-- (7,2);
\draw (7,2)-- (6,2);
\draw (6,2)-- (6,1);
\draw (6,1)-- (7,1);
\draw (8,1)-- (8,2);
\draw (8,2)-- (7,2);
\draw (7,2)-- (7,1);
\draw (7,1)-- (8,1);
\draw (6,3)-- (7,3);
\draw (7,3)-- (7,4);
\draw (7,4)-- (6,4);
\draw (6,4)-- (6,3);
\draw (6,3)-- (6,2);
\draw (6,2)-- (7,2);
\draw (7,2)-- (7,3);
\draw (7,3)-- (6,3);
\draw (7,4)-- (7,3);
\draw (7,3)-- (8,3);
\draw (8,3)-- (8,4);
\draw (8,4)-- (7,4);
\draw (6,4)-- (7,4);
\draw (7,4)-- (7,5);
\draw (7,5)-- (6,5);
\draw [line width=0.4pt] (6,5)-- (6,4);
\draw (8,5)-- (7,5);
\draw (7,5)-- (7,4);
\draw (7,4)-- (8,4);
\draw (8,4)-- (8,5);
\draw (8,4)-- (9,4);
\draw (9,4)-- (9,5);
\draw (9,5)-- (8,5);
\draw (8,5)-- (8,4);
\end{tikzpicture}
\quad\quad\quad\quad\quad\quad
\begin{tikzpicture}[line cap=round,line join=round,>=triangle 45,x=0.8cm,y=0.8cm]
\clip(4.24,0.62) rectangle (9.28,4.7);
\fill[fill=black,fill opacity=0.1] (5,1) -- (6,1) -- (6,2) -- (5,2) -- cycle;
\fill[fill=black,fill opacity=0.1] (7,1) -- (7,2) -- (6,2) -- (6,1) -- cycle;
\fill[fill=black,fill opacity=0.1] (6,3) -- (6,2) -- (7,2) -- (7,3) -- cycle;
\fill[fill=black,fill opacity=0.1] (5,2) -- (6,2) -- (6,3) -- (5,3) -- cycle;
\fill[fill=black,fill opacity=0.1] (7,2) -- (8,2) -- (8,3) -- (7,3) -- cycle;
\fill[fill=black,fill opacity=0.1] (7,3) -- (8,3) -- (8,4) -- (7,4) -- cycle;
\fill[fill=black,fill opacity=0.1] (8,3) -- (9,3) -- (9,4) -- (8,4) -- cycle;
\draw (5,1)-- (6,1);
\draw (6,1)-- (6,2);
\draw (6,2)-- (5,2);
\draw (5,2)-- (5,1);
\draw (7,1)-- (7,2);
\draw (7,2)-- (6,2);
\draw (6,2)-- (6,1);
\draw (6,1)-- (7,1);
\draw (6,3)-- (6,2);
\draw (6,2)-- (7,2);
\draw (7,2)-- (7,3);
\draw (7,3)-- (6,3);
\draw (5,2)-- (6,2);
\draw (6,2)-- (6,3);
\draw (6,3)-- (5,3);
\draw (5,3)-- (5,2);
\draw (7,2)-- (8,2);
\draw (8,2)-- (8,3);
\draw (8,3)-- (7,3);
\draw (7,3)-- (7,2);
\draw (7,3)-- (8,3);
\draw (8,3)-- (8,4);
\draw (8,4)-- (7,4);
\draw (7,4)-- (7,3);
\draw (8,3)-- (9,3);
\draw (9,3)-- (9,4);
\draw (9,4)-- (8,4);
\draw (8,4)-- (8,3);
\end{tikzpicture}
\caption{}\label{P}
\end{figure}

Let $\MP$ be a polyomino, and let $K$ be a field. We denote by $S$ the polynomial over $K$ with variables $x_{ij}$ with $(i,j)\in V(\MP)$. A $2$-minor $x_{ij}x_{kl}-x_{il}x_{kj}\in S$ with $i<k$ and $j<l$ is called an {\em inner minor} of $\MP$ if all the cells $[(r,s),(r+1,s+1)]$ with $i\leq r\leq k-1$ and $j\leq s\leq l-1$ belong to $\MP$. The ideal $I_\MP\subset S$ generated by all inner minors of $\MP$ is called the {\em polyomino ideal} of $\MP$. We also set $K[\MP]=S/I_\MP$. It has been shown in \cite{Q} that $K[\MP]$ is a domain, and hence a toric ring,  if $\MP$ is convex. A toric parametrization of $K[\MP]$ will be given in the following proof.

\begin{Theorem}
\label{polyomino}
Let $\MP$ be  a convex polyomino. Then $k[\MP]$ is inseparable.
\end{Theorem}

\begin{proof}
Set
$
A_\MP=\{h_i: (i,j)\in V(\MP) \text{ for some } j\in \mathbb{Z}_+\}
$
 and
$
B_\MP=\{v_j: (i,j)\in V(\MP) \text{ for some } i\in \mathbb{Z}_+\}.
$
We associate with  $\MP$ a bipartite graph $G(\MP)$ such that $V(G(\MP))=A_P\cup B_\MP$ and $E(G(\MP))=\{\{h_i,v_j\}:  (i,j)\in V(\MP)\}$. Figure~\ref{G} shows a polyomino and its associated bipartite graph.

\begin{figure}[ht!]
\definecolor{qqqqff}{rgb}{0,0,1}
\begin{tikzpicture}[line cap=round,line join=round,>=triangle 45,x=0.85cm,y=0.85cm]
\clip(-1.84,-0.44) rectangle (3.8,3.76);
\fill[fill=black,fill opacity=0.1] (0,0) -- (1,0) -- (1,1) -- (0,1) -- cycle;
\fill[fill=black,fill opacity=0.1] (0,1) -- (1,1) -- (1,2) -- (0,2) -- cycle;
\fill[fill=black,fill opacity=0.1] (1,1) -- (2,1) -- (2,2) -- (1,2) -- cycle;
\fill[fill=black,fill opacity=0.1] (2,2) -- (3,2) -- (3,3) -- (2,3) -- cycle;
\fill[fill=black,fill opacity=0.1] (0,0) -- (0,1) -- (-1,1) -- (-1,0) -- cycle;
\fill[fill=black,fill opacity=0.1] (2,1) -- (3,1) -- (3,2) -- (2,2) -- cycle;
\draw (0,0)-- (1,0);
\draw (1,0)-- (1,1);
\draw (1,1)-- (0,1);
\draw (0,1)-- (0,0);
\draw (0,1)-- (1,1);
\draw (1,1)-- (1,2);
\draw (1,2)-- (0,2);
\draw (0,2)-- (0,1);
\draw (1,1)-- (2,1);
\draw (2,1)-- (2,2);
\draw (2,2)-- (1,2);
\draw (1,2)-- (1,1);
\draw (2,2)-- (3,2);
\draw (3,2)-- (3,3);
\draw (3,3)-- (2,3);
\draw (2,3)-- (2,2);
\draw (0,0)-- (0,1);
\draw (0,1)-- (-1,1);
\draw (-1,1)-- (-1,0);
\draw (-1,0)-- (0,0);
\draw (2,1)-- (3,1);
\draw (3,1)-- (3,2);
\draw (3,2)-- (2,2);
\draw (2,2)-- (2,1);
\draw (-1.64,0.1) node[anchor=north west] {\begin{tiny}(1,1)\end{tiny}};
\draw (0.64,0.12) node[anchor=north west] {\begin{tiny}(3,1)\end{tiny}};
\draw (2.62,1.1) node[anchor=north west] {\begin{tiny}(5,2)\end{tiny}};
\draw (0.84,1.1) node[anchor=north west] {\begin{tiny}(3,2)\end{tiny}};
\draw (-1.68,1.5) node[anchor=north west] {\begin{tiny}(1,2)\end{tiny}};
\draw (-0.68,2.5) node[anchor=north west] {\begin{tiny}(2,3)\end{tiny}};
\draw (-0.9,1.5) node[anchor=north west] {\begin{tiny}(2,2)\end{tiny}};
\draw (1.68,1.1) node[anchor=north west] {\begin{tiny}(4,2)\end{tiny}};
\draw (-0.6,0.1) node[anchor=north west] {\begin{tiny}(2,1)\end{tiny}};
\draw (0.3,2.5) node[anchor=north west] {\begin{tiny}(3,3)\end{tiny}};
\draw (1.1,2.5) node[anchor=north west] {\begin{tiny}(4,3)\end{tiny}};
\draw (2.85,2.5) node[anchor=north west] {\begin{tiny}(5,3)\end{tiny}};
\draw (1.56,3.5) node[anchor=north west] {\begin{tiny}(4,4)\end{tiny}};
\draw (2.54,3.5) node[anchor=north west] {\begin{tiny}(5,4)\end{tiny}};
\begin{scriptsize}
\end{scriptsize}
\end{tikzpicture}
\quad\quad\quad\quad
\begin{tikzpicture}[line cap=round,line join=round,>=triangle 45,x=1.2cm,y=1.2cm]
\clip(7,0.5) rectangle (12.92,4.02);
\draw (8,3)-- (8.62,1.04);
\draw (8,3)-- (9.52,0.98);
\draw (9,3)-- (8.62,1.04);
\draw (9,3)-- (9.52,0.98);
\draw (9,3)-- (10.52,0.98);
\draw (10,3)-- (8.62,1.04);
\draw (10,3)-- (9.52,0.98);
\draw (10,3)-- (10.52,0.98);
\draw (11,3)-- (9.52,0.98);
\draw (11,3)-- (10.52,0.98);
\draw (11,3)-- (11.52,1);
\draw (12,3)-- (9.52,0.98);
\draw (12,3)-- (10.52,0.98);
\draw (12,3)-- (11.52,1);
\draw (7.8,3.4) node[anchor=north west] {\begin{scriptsize}$h_1$\end{scriptsize}};
\draw (8.8,3.4) node[anchor=north west] {\begin{scriptsize}$h_2$\end{scriptsize}};
\draw (9.78,3.4) node[anchor=north west] {\begin{scriptsize}$h_3$\end{scriptsize}};
\draw (10.76,3.4) node[anchor=north west] {\begin{scriptsize}$h_4$\end{scriptsize}};
\draw (11.79,3.4) node[anchor=north west] {\begin{scriptsize}$h_5$\end{scriptsize}};
\draw (8.41,1.01) node[anchor=north west] {\begin{scriptsize}$v_1$\end{scriptsize}};
\draw (9.34,1.01) node[anchor=north west] {\begin{scriptsize}$v_2$\end{scriptsize}};
\draw (10.32,1.01) node[anchor=north west] {\begin{scriptsize}$v_3$\end{scriptsize}};
\draw (11.32,1.01) node[anchor=north west] {\begin{scriptsize}$v_4$\end{scriptsize}};
\begin{scriptsize}
\fill [color=black] (0,0) circle (1.5pt);
\fill [color=black] (1,0) circle (1.5pt);
\fill [color=black] (2,1) circle (1.5pt);
\fill [color=black] (3,2) circle (1.5pt);
\draw[color=black] (3.14,2.28) node {$L$};
\fill [color=black] (3,1) circle (1.5pt);
\draw[color=black] (3.14,1.28) node {$J$};
\fill [color=black] (8,3) circle (1.5pt);
\fill [color=black] (9,3) circle (1.5pt);
\fill [color=black] (10,3) circle (1.5pt);
\fill [color=black] (11,3) circle (1.5pt);
\fill [color=black] (12,3) circle (1.5pt);
\fill [color=black] (8.62,1.04) circle (1.5pt);
\fill [color=black] (9.52,0.98) circle (1.5pt);
\fill [color=black] (10.52,0.98) circle (1.5pt);
\fill [color=black] (11.52,1) circle (1.5pt);
\end{scriptsize}
\end{tikzpicture}
\caption{}\label{G}
\end{figure}
We let $K[G(\MP)]$ be the subring of the polynomial ring $T = K[A_\MP\cup B_\MP]$  generated by the monomials $h_iv_j$ with  $\{h_i,v_j\}\in E(G(\MP))$. In other words, $K[G(\MP)]$ is the edge ring of the bipartite graph $G(\MP)$. Let, as above,  $S=K[x_{ij}: (i,j)\in V(\MP)]$. As shown in \cite{Q}, $I_\MP$ is the kernel of the $K$-algebra homomorphism $S\to K[G(\MP)]$ with  $x_{ij}\mapsto h_iv_j$. Thus $K[\MP]\iso K[G(\MP)]$, and $K[G(\MP)]$ is the desired toric parametrization.  It is known from \cite{OH} that $I_\MP$ is generated by the binomials corresponding  to the cycles in $G(\MP)$.

By using Corollary~\ref{unique} it is enough to show that  for any cycle $C$ of $G(\MP)$ which has a unique chord, say $e=\{h_i,v_j\}$,  there is a  crossing path chord of $C$ with respect to  $e$.  Since $G(\MP)$ is a bipartite graph, $C$ is an even cycle, and also $|C|>4$ because $C$ has a chord. Since every induced cycle of $G(\MP)$ is a $4$-cycle and since $C$ has only one chord, $C$ must be a $6$-cycle. Assume that the vertices of  $C$ are $h_i, v_{k_1}, h_{{\ell}_1},v_j,h_{{\ell}_2}, v_{k_2}$, listed counterclockwise, and the chord of $C$ is  $e=(h_i,v_j)$ as above.
With the notation introduced, it follows that
\[
(i,j),(i,k_2),(\ell_2,k_2),(\ell_2,j),(\ell_1,j),(\ell_1,k_1,),(i,k_1)
\]
are  vertices of $\MP$. We consider the following cases.

Suppose first that $(\ell_1-i)(\ell_2-i)>0$. Without loss of generality,  we may assume  $\ell_2>\ell_1>i$.  Then, since $\MP$ is convex and $(i,k_2)$ and $(\ell_2,k_2)$ are both vertices of $\MP$, we have $(\ell_1,k_2)$ is a vertex of $\MP$. It follows that $\{h_{\ell_1},v_{k_2}\}$ is an edge of $G(\MP)$ which is  a   chord  of $C$, contradicting our assumption that $C$ has a unique chord.
Similarly the case that $(k_1-j)(k_2-j)>0$ is also not possible.

It remains to consider the case when $(\ell_1-i)(\ell_2-i)<0$ and $(k_1-j)(k_2-j)<0$.  Without loss of generality we may assume that $\ell_1<i<\ell_2$ and $k_1<j<k_2$. Then either $(i-1,j+1)$ or $(i+1,j-1)$ is a vertex of $\MP$ by the connectedness and convexity of $\MP$.

We may assume that $(i-1,j+1)\in V(\MP)$. Note that $(i-1,k_1)$ and $(\ell_2,j+1)$ belong to $V(\MP)$. Thus  we obtain the path $v_{k_1}, h_{i-1},v_{j+1}, h_{\ell_2}$ in $G(\MP)$ which is a crossing path chord of  $C$ with respect to  $e$.
\end{proof}

\section{On the semi-rigidity  of bipartite graphs}

We say that $R$ is {\em semi-rigid} if $T^1(R)_a=0$ for all $a\in \ZZ H$ with   $-a\in H$.  In this subsection we consider this weak form of rigidity which however is stronger than inseparability.

We again let $G$ be a finite bipartite graph on the vertex set $[m]$ with edge set $E(G)=\{e_1,e_2,\ldots, e_n\}$. The edge ring of $G$ is the toric ring $K[H]$ whose generators are the elements  $h_i=\sum_{j\in V(e_i)}\delta_j$, $i=1,\ldots,n$. Here  $\delta_1,\ldots, \delta_m$ is the canonical basis of $\ZZ^m$.  As above we may assume that each edge of $G$ belongs to a cycle and that $C_1,C_2,\ldots C_s$ is the set of cycles of $G$ and where $C_1,\ldots,C_{s_1}$ is the set of induced cycles of $G$.

Let $C_i$ be one of these cycles with  edges $e_{i_1},e_{i_2},\ldots,e_{i_{2t}}$ labeled counterclockwise.  Two distinct edges $e$ and $e'$ of $C_i$ are said to be of the {\em same parity in} $C_i$ if   $e=e_{i_j}$ and $e'=e_{i_k}$ with $j-k$ an even number.

\begin{Lemma}
\label{terrible}
Let $a=-h_j-h_k$, and  let $i\in [s_1]$. Then
$i\in \mathcal{F}'_a$, if and only if $e_j$ and $e_k$ have the same parity in $C_i$. Moreover, if  $\mathcal{F}'_a\neq \emptyset$, then $KL_a=KL_{-h_j}+KL_{-h_k}$.
\end{Lemma}
\begin{proof}Since $i\in [s_1]$, the cycle $C_i$ is an induced cycle. Let  $e_{i_1},e_{i_2},\ldots,e_{i_{2t}}$ be the edges of $C_i$ labeled counterclockwise. Then $h(v_i)= \sum_{k=1}^t h_{i_{2k-1}}= \sum_{k=1}^th_{i_{2k}}$. Thus if $e_j$ and $e_k$ have the same parity in $C_i$, it follows that $h_j$ and $h_k$ belong to either one of the above summands, so that $a+h(v_i)\in H$. This shows that $i\in \mathcal{F}'_a$. Conversely, suppose that  $i\in \mathcal{F}'_a$.  Let $h_j= \delta_{j_1}+\delta_{j_2}$ and   $h_k= \delta_{k_1}+\delta_{k_2}$. For simplicity, we may assume that  $\delta_1,\ldots,\delta_{2t}$ correspond to the vertices of $C_i$ and that the edges of $C_i$ correspond to the elements $\delta_{2t}+\delta_1$ and $\delta_i+\delta_{i+1}$ for $i=1,\ldots,2t-1$. Then $h(v_i)=  \delta_1+\cdots+\delta_{2t}$ and
\begin{eqnarray}
\label{inh}
a+h(v_i)=-\delta_{j_1}-\delta_{j_2}-\delta_{k_1}-\delta_{k_2}+\delta_1+\ldots+\delta_{2t}\in H.
\end{eqnarray}
In general, let $h\in H$, $h=\sum_{i=1}^mz_i\delta_i$ with $z_i\in \ZZ$. Then it follows that $z_i\geq 0$ for all $i$.
Hence it follows from (\ref{inh}) that $e_j$ and $e_k$ are edges of $C_i$ with $V(e_j)\cap V(e_k)=\emptyset$ (that is, the vertices $j_1,j_2,k_1,k_2$ are pairwise different),  and that $a+h(v_i)$ is the sum of all $\delta_i$, $i=1,\ldots, 2t$ with $i\neq j_1,j_2, k_1,k_2$.  Suppose the edges $e_j$ and $e_k$  do not have the same parity in $C_i$. Then $a+h(v_i)$ is the sum of  $S_1$ and $S_2$,  where each of  $S_1$ and $S_2$  consists of an odd sum of $\delta_i$. Hence none of these summands  belongs to $H$. Since $S_1+S_2\in H$, there exists  a summand $\delta_{r_1}$ in $S_1$ and a summand $\delta_{r_2}$ in $S_2$ such that $\delta_{r_1}+\delta_{r_2}\in H$. This implies that $\{r_1,r_2\}\in E(C_i)$ because $C_i$ has no chord.  However this is not possible.  Indeed,  if $\{r_1,r_2\}\in E(C_i)$, then $r_2 \equiv r_1+1\mod 2t$.  But this is not the case.

Next we show that $KL_a=KL_{-h_j}+KL_{-h_k}$ if
$\mathcal{F}'_a\neq \emptyset$. Note that  $\MF'_a\subseteq \MF'_{-h_j}\cap \MF'_{-h_k}$, we have $KL_{-h_j}+KL_{-h_k}\subseteq KL_a$ by Lemma~\ref{basic}.
In order to obtain the desired equality,  we only need to show that $v_i\in KL_{-h_j}+KL_{-h_k}$ for each $i\in  (\MF'_{-h_j}\cap \MF'_{-h_k})\setminus \MF'_a$.

Let $i\in  (\MF'_{-h_j}\cap \MF'_{-h_k})\setminus \MF'_a$. Since $\MF'_a\neq \emptyset$, there exists an induced  cycle, say  $C$,  such that $e_j$ and $e_k$ have the same parity in $C$. We may assume that $V(C)=[2t]$ and $E(C)=\{\{1,2\},\{2,3\},\ldots,\{2t-1,2t\},\{2t,1\}\}$, and  that $e_j=\{1,2\}$ and $e_k=\{2k-1,2k\}$  with $1<k\leq t$. Since $e_j,e_k$ do not have the same  parity in $C_i$, we can assume without loss of generality that $E(C_i)$ is
 \begin{eqnarray*}
 &&\{\{1,2\},\{2,i_1\},\{i_1,i_2\},\ldots,\{i_{2h},i_{2h+1}\},\{i_{2h+1},2k\},\{2k,2k-1\}\}\\
&&\cup \{\{2k-1,i_{2h+2}\},\ldots, \{i_{2\ell}, i_{2\ell+1}\},\{i_{2\ell+1},1\}\}.
 \end{eqnarray*}
 Then we have even closed walks
  \[
  W_1\: 2 ,3, \ldots, (2k-1),2k,i_{2h+1},i_{2h}\ldots,i_1,2
  \]
and
\[
W_2: 1,2,3,\ldots,(2k-1),i_{2h+2},\ldots,i_{2\ell+1},1.
\]
 Let $w_1=v(W_1)$ and $w_2=v(W_2)$. Since the vertex $1$ belongs to $e_j$ but is not a vertex of $W_1$, Lemma~\ref{walk} implies that    $w_1\in KL_{-h_j}$. Similarly it follows that $w_2\in  KL_{-h_k}$.   Since  $v_i$ differs at most by  a sign from either  $w_1-w_2$ or $w_1+w_2$, it follows that  $v_i\in KL_{-h_j}+KL_{-h_k}$, as required.
\end{proof}

\begin{Lemma} Suppose that $\MF'_{-h_j}\neq \MF'_{-h_k}$. Then $KL_{-h_j}\neq KL_{-h_k}$.
\label{not}
\end{Lemma}

\begin{proof} Let $i\in \MF'_{-h_j}\setminus \MF'_{-h_k}$. Then $v_i\in KL_{-h_k}$ and $v_i(j)\neq 0$, since $e_j$ is an edge  of $C_i$. However the vectors $v$ which  belong to $KL_{-h_j}$ have the property that  $v(j)=0$. Hence $v_i\in KL_{-h_k}\setminus KL_{-h_j}$,  and this implies $KL_{-h_j}\neq KL_{-h_k}$.
\end{proof}

\begin{Corollary}

\label{semi} Assume that $K[G]$ is inseparable.
Let $a=-h_j-h_k$. Then
\[
\dim_K KL_a=\dim_KKL-1\text{ if $\mathcal{F}'_{-a}\neq \emptyset$ and $\mathcal{F}'_{-h_j}=\mathcal{F}'_{-h_k}$}.
\]
Otherwise, $\dim_K KL_a=\dim_KKL$.

\end{Corollary}

\begin{proof} Since we assume that $G$ is inseparable, it follows from Corollary~\ref{F1} and Proposition~\ref{F2}  that $\dim_KKL-\dim_KKL_{-h_j}=\dim_K(\Im\delta^*)_{-h_j}$.  Since by assumption each edge of $G$ belongs to a cycle, it follows that $\dim_K(\Im\delta^*)_{-h_j}=1$.  Thus $\dim_KKL_{-h_j}= \dim_KKL-1$. Similarly, $\dim_KKL_{-h_k}= \dim_KKL-1$.  If $\mathcal{F'}_{-h_j}=\mathcal{F'}_{-h_k}$, then  $KL_{-h_j}=KL_{-h_k}$,  and if moreover, $\mathcal{F'}_{-a}\neq \emptyset$,  then together with  Lemma~\ref{terrible} we have $\dim_K KL_a=\dim_KKL-1$, as desired.

 Otherwise, there are two cases to consider.
If $\mathcal{F'}_{-a}= \emptyset$, then $KL_a=KL$,  by the definition of $KL_a$ and by Lemma~\ref{basic}. If $\mathcal{F'}_{-a}\neq \emptyset$ and $\mathcal{F'}_{-h_j}\neq \mathcal{F'}_{-h_k}$, then $KL_a=KL_{-h_j}+KL_{-h_k}=KL$,  using Lemma~\ref{terrible} together with Lemma~\ref{not}.
\end{proof}

\begin{Theorem} Let $G$ be a bipartite graph such that $R=K[G]$ is inseparable. Then the following statements are equivalent:

\label{semi-rigid}
\begin{enumerate}
\item[(a)] $K[G]$ is not semi-rigid;

 \item[(b)] there exist edges $e,f$ and an induced cycle $C$ such that $e,f$ have the same parity in $C$ and for any  other induced cycle $C'$, $e\in E(C')$ if and only if $f\in E(C')$.

 \end{enumerate}

 \end{Theorem}

 \begin{proof}

 (b)\implies (a): Let $a=-g-h$, where  $g$ and $h$ are vectors in $H$ corresponding to the edges $e$ and $f$ respectively. Then  $\dim_K KL_a=\dim_K KL-1$ by  Corollary~\ref{semi}. Note that $\MG_a=\emptyset$, we have $(\Im \delta^*)_a=0$. Therefore $T^1(R)_a\neq 0$ by Corollary~\ref{rigid}, and in particular, $R$ is not semirigid.

(a)\implies (b): By assumption, there exists $a=\sum _{i\in [n]}-a_ih_i\in \ZZ H$ with $a_i\geq 0$ for $i=1,\cdots,n$ such that $T^1(R)_a\neq 0$.
Note that $a_i\in \{0,1\}$, for otherwise, $\MF'_a=\emptyset$ and so $KL_a=KL$. In particular $T^1(R)_a= 0$, a contradiction. Since $R$ is inseparable, it follows that  $|\{i\:\; a_i\neq 0\}|\geq 2$. If  $|\{i\:\; a_i\neq 0\}|= 2$, then $a=-h_k-h_j$ for some $1\leq i\neq j\leq n$. Therefore,   $\MF'_a\neq \emptyset$ and $\MF'_{-h_j}=\MF'_{-h_k}$ by Corollary~\ref{F1} and Corollary~\ref{semi}.

Let $e$ and $f$ be the  edges corresponding to the vectors  $h_j$ and $h_k$,  respectively. Then, since $\MF'_a\neq \emptyset$,  there exists an induced  cycle $C$ of $G$ such that $e$ and $f$ have the same parity in $C$, by Lemma~\ref{terrible}.  Moreover,  $\MF'_{-h_j}=\MF'_{-h_k}$ implies that for any induced cycle $C'$ of $G$, $e\in E(C')$ if and only if $f\in E(C')$.

Now suppose that  $|\{i\:\; a_i\neq 0\}|\geq 3$. Then there exists $j$ and $k$ with $a_j\neq 0$ and $a_k\neq 0$, and we set  $b=-h_j-h_k$. Note that  $\MF'_a\subseteq \MF'_b$. This implies that $KL_b\subseteq KL_a$.  Therefore, since   $(\Im \delta^*)_a=(\Im \delta^*)_b=0$, we have $T^1(R)_b\neq 0$, and we are in the previous case.
 \end{proof}

 \begin{Corollary}
\label{polyomino1}
Let $\MP$ be a convex polyomino. Then $K[\MP]$ is semi-rigid if and only if $\MP$ contains more than one cell.
\end{Corollary}

\begin{proof} Assume that $\MP$ contains a unique cell. Then $G(\MP)$ is a square and it is not semi-rigid by Theorem~\ref{semi-rigid}.

\medskip
Conversely, assume that $K[\MP]$ is not semi-rigid. Then there exist two edges $e,f$ and an induced cycle $C$ of $G(\MP)$ satisfying the condition (b) in Theorem~\ref{semi-rigid}. Let $(i,j)$ and $(k,\ell)$ be vertices of $\MP$ corresponding to the edge $e$ and $f$, respectively. Then the  two edges of $C$ other than $e$ and $f$ correspond to the vertices $(i,\ell)$ and $(k,j)$ of $\MP$. It follows that $k\neq i$ and $\ell\neq j$.   Without loss of generality, we may assume that $k>i$ and $\ell>j$. Then $(i+1,j+1)\in V(\MP)$. Let $C'$ be the induced cycle of $G(\MP)$ corresponding to the cell $[(i,j),(i+1,j+1)]$ of $\MP$. Since $C'$ contains the edge $e$,  $C'$ must contain $f$ by the condition (b) and thus $k=i+1$ and $\ell=j+1$. We claim that $[(i,j),(i+1,j+1)]$ is the only cell of $\MP$. Suppose that this is not the case. Then we let $C_t$, $t=1,2,3,4$ be four cells  which share a common edge with the cell $[(i,j),(i+1,j+1)]$. Note that $\MP$ contains at least one of the $C_t$. Indeed, since $\MP$ is connected and  since by assumption $\MP$ contains a cell $C$  different from $[(i,j),(i+1,j+1)]$, there   exists a path in $\MP$ between the cell $[(i,j),(i+1,j+1)]$ and $C$. This path must contain one of the $C_t$. However $V(C_t)$ contains exactly one of the two vertices $(i,j)$ and $(i+1,j+1)$ for $t=1,\ldots,4$. In other words, there exists an induced cycle of $G(\MP)$ which contains exactly one of the edges $e$ and $f$. This is contradicted to the condition (b) and thus our claim has been proved.
\end{proof}

\section{Classes of bipartite graphs which are semi-rigid or rigid}

 As an example of an application of  Formula~(\ref{finalformua}),  we will   show  that the edge ring of a large complete bipartite graph with one edge removed is rigid.

\medskip
Let $G_{k,m-k}$ be a bipartite graph on parts $U=\{1,\ldots,k\}$ and $V=\{k+1,\ldots,m\}$ with edge set $$E(G_{k,m-k})=\{\{i,j\}\:\; i\in U, j\in V, \{i,j\}\neq \{1,m\} \}.$$ Thus $G_{k,m-k}$ is obtained from the complete bipartite graph $K_{k,m-k}$ by deleting  one of its edges.

 Our main result of this section is the following:

\medskip
\begin{Proposition} Let $R$ be the edge ring of $G_{k,m-k}$.

\label{main4}
\begin{enumerate}
\item[(a)] If $k=m-k=3$, then $R$ is inseparable, but not rigid.
\item[(b)] If $m-k\geq k\geq  4$, then $R$ is rigid.
\end{enumerate}

\end{Proposition}

We need some preparations. First, we determine when an element in $\ZZ^m$ belongs to $H$ and $\ZZ H$, where $H=H(G_{k,m-k})$. For this, we introduce some notation, which is used throughout this section.

\medskip
Let $a=(a_1,\ldots,a_m)\in \ZZ^{m}$. We set $$a_U=\sum_{i\in U}a_i\mbox{\qquad and\qquad } a_V=\sum_{i\in V}a_i.$$

We also set $$\ell(a)=a_1+a_{m} \mbox{\qquad and\qquad } r(a)=\sum_{i=2}^{m-1} a_i.$$

Recall that for an cycle $C$ we use $V(C)$ for its vertex set and $v(C)$ for the corresponding vector of $C$, which is unique up to sign. Note that the degree $h(v(C))$ of $v(C)$ is $\sum_{i\in V(C)}\delta_i$.  For any edge $e=\{i,j\}\in E(G)$ we use $h(e)$ to denote the vector $\delta_i+\delta_j\in  \ZZ^{m}$

\begin{Lemma}

\label{cri} Let  $H=H(G_{k,m-k})\subset \ZZ^m$ . Then for any $a\in \ZZ^m$,

\begin{enumerate}
\item[(1)] $a\in \ZZ H$ if and only if $a_U=a_V$.
\item[(2)] The following conditions are equivalent:
\begin{enumerate}
\item[(i)] $a\in H$;
\item[(ii)] $a_U=a_V$, $\ell(a)\leq r(a)$ and $a_i\geq 0$ for all $i=1,\ldots,m$.
\end{enumerate}

\item[(3)] Let $a\in \ZZ H$ with $a_i\geq 0$ for all $i\in [m]$. Then  either $a\in H$ or $a=b+k(\delta_1+\delta_{m})$, where $k\geq 1$ and $b\in H$ with $\ell(b)=r(b)$.
\end{enumerate}
\end{Lemma}

\begin{proof} (1) It is clear that $a_U=a_V$ if $a\in \ZZ H$. For the converse, first note that $\delta_1+\delta_m=(\delta_1+\delta_{k+1})+(\delta_2+\delta_m)-(\delta_2+\delta_{k+1})\in \ZZ H$. Then the result follows by induction on $|a_U|$.

(2) (i)\implies (ii): Note that $\ell(h(e))\leq r(h(e))$ for any  $e\in E(G)$ since $\{1,m\}\notin E(G)$. Now given $a\in H$. Then $a=\sum_{e\in E(G)} c_eh(e)$, where $c_e$ is a non-negative integer for each $e\in E(G)$. It follows that  $\ell(a)=\sum_{e\in E(G)}c_e\ell(h(e))\leq \sum_{e\in E(G)}c_er(h(e))=r(a)$, as required.

(ii)\implies (i): We use induction on $\ell(a)$. If $\ell(a)=0$, we see that $a\in H$ by induction on $a_U$. Assume that $\ell(a)>0$. Without restriction we may further assume  that $a_1\geq a_{m}$. Then $a_V-a_m\geq a_U-a_1$. Note that $r(a)\geq \ell(a)>0$, one has $a_1>0$ and $a_v-a_m>0$.
 Hence there exists  an even number $k+1\leq j\leq m-1$ with $a_j>0$. Since $b:=a-(\delta_1+\delta_j)\in H$ by induction, it follows that $a=b+(\delta_1+\delta_j)\in H$.

(3) Suppose that $a\notin H$. Then $\ell(a)> r(a)$, by (2). Note that  $a_U=a_V$ by (1),  we have   $\ell(a)-r(a)=(a_U+a_V)-2r(a)$ is an even number, say $2k$.  It follows that $a_1\geq k$ and $a_m\geq k$, since $a_1+a_m=a_U+k$. Set $b=a-k(\delta_1+\delta_m)$. Then $b_i\geq 0$ for $i\in [m]$ and $\ell(b)=r(b)$. In particular, $b\in H$ by (2), as required.
\end{proof}

 In the proof of the following lemma we use a well-known fact from graph theory: if $F$ is a subset of the edge set $E(G)$ of a connected graph and $F$ contains no cycle, then there is a spanning tree $\Gamma$ of $G$ such that $F\subseteq E(\Gamma)$. Here a spanning tree of a connected graph $G$ means that a subgraph of $G$ which is tree having the same vertex  set as $G$.

\begin{Lemma} Let $G$ be a connected graph and denote $H=H(G)$. Then for $a\in \ZZ H$ such that $\{e\in E(G)\:\; a+h(e)\notin H\}$ contains no cycle, we have $\dim_K KL=\dim_K D_a$. In particular, $T^1(K[G])_a=0$.

\label{spanning}

\end{Lemma}
\begin{proof}
Let $F=\{e\in E(G)\:\; a+h(e)\notin H\}$. Since $F$ contains no cycle, there exists a spanning tree $\Gamma$ of $G$ such that $F\subset E(\Gamma)$. Without loss of generality we assume that $E(G)\setminus E(\Gamma)=\{e_1,\ldots,e_r\}$, where $r=mk-k^2-m$. Note that $h(e_i)+a\notin H$ for each $i=1,\ldots,r.$

For each $i=1,\ldots,r$, $\Gamma+e_i$ contains a unique induced cycle, say $C_i$. Let $v_i=v(C_i)$, the vector corresponding to the cycle $C_i$ for $i=1,\ldots,k$. Then  for all $i=1,\ldots,r$ we have $v_i(i)\in \{\pm1\}$ and $v_i(j)=0$ if $j\neq i$ and $1\leq j\leq r$.   It follows that $\dim_K D_a\geq r$ since $(v_1(i),\ldots,v_r(i),\ldots,v_{s}(i))\in D_a$ for  $i=1,\ldots,r$. Here $s$ is the number of induced cycles of $G$. On the other hand, $\dim_K D_a\leq \dim_K KL$ and $\dim_k KL=|E(G)|-|V(G)|+1=r$. Hence $\dim_KKL=\dim_KD_a$ and  $T^1(K[G])_a=0$ by Proposition~\ref{F2}.
\end{proof}

\begin{proof}[Proof of Proposition~\ref{main4}]
(a) Since $G_{3,3}$ is inseparable, we have $K[G_{3,3}]$ is inseparable by Corollary~\ref{unique}.

   Let $a=\delta_6-\delta_1-\delta_4-\delta_5\in \ZZ H$. Then $KL_a$ is spanned by the vectors corresponding to the cycles $C_1: 2,4,3,5$, $C_2: 2,5,3,6$ and $C_3: 2,4,3,6$. This implies that $\dim_K KL_a=2$. Since $\dim_K D_a=0$ and $\dim_KKL=3$, we have $T^1(R)_a= 1\neq 0$. In particular, $R$ is not rigid, as required.
\medskip

(b) Assume that $m\geq 4$ and $m-k\geq 4$. Denote $G_{k,m-k}$ by $G$ and $K[G_{k,m-k}]$ by $R$.  We want to prove that $T^1(R)_a=0$ for each $a\in \ZZ H\subset \ZZ^m$. We distinguish the  following cases.

\noindent
Case 1 :  $a_i\geq 0$ for all $i\in [m]$. By Lemma~\ref{cri}, either $a\in H$ or $a=b+k(\delta_1+\delta_{m})$, where $k\geq 1$ and $b\in H$ with $\ell(b)=r(b)$. If $a\in H$, then $T^1(R)_a=0$, see  Corollary~\ref{inH}. If $a=b+k(\delta_1+\delta_{2n})$ with $k=1$, then for any edge $e=\{i,j\}$ with $e\cap \{1,m\}=\emptyset$, we have $a+\delta_i+\delta_j\in H$ by Lemma~\ref{cri}. It follows that
$\{e\in E(G)\:\; a+h(e)\notin H\}$ contains no cycle, and so $T^1(R)_a=0$ by Lemma~\ref{spanning}. If $k=2$,  then for any induced cycle $C$,  $a+h(v(C))\in H$ if and only if $V(C)\cap \{1,m\}=\emptyset$. This follows from  Lemma~\ref{cri} and the fact that any induced cycle of $G$ is a $4$-cycle. To prove $KL=KL_a$, we have to show if $V(C)\cap \{1,m\}=\emptyset$, then $v(C)\in KL_a$. Given an induced cycle $C:i_1,i_2,i_3,i_4$ with $V(C)\cap \{1,m\}=\emptyset$, where $\{i_1,i_3\}\subset U$  and $\{i_2,i_4\}\subset V$. Then we obtain two  cycles $C_1$: $i_1,i_2,i_3,1$ and $C_2:$  $i_3,i_4,i_1,1$. Note that $v(C_1),v(C_2)\in KL_a$ and $v(C)$ is a linear combination of $v(C_1),v(C_2)$, we have  $KL=KL_a$ and so $T^1(R)_a=0$. If $k\geq 3$, then  for any induced cycle $C$, one  has $a+h(v(C))\notin H$ by Lemma~\ref{cri} and so $KL_a=KL$. In particular, $T^1(R)_a=0$.

\vspace{2mm}\noindent {\bf Remark:}  If $a_i\leq -2$, then $\MF'_a=\emptyset$ and so  $T^1(R)_a=0$. In the following cases, we always assume that $a_i=-1$ if $a_i<0$. \vspace{2mm}

\noindent Case 2:  There exists a unique $i\in [m]$ with $a_i<0$.  Then $a_i=-1$. By symmetry, we only need to consider the cases when $i=1$ and when  $i=2$.

We first assume that $i=1$. Since $a_U=a_V$, there exists  $1\neq j\in U$ such that $a_j>0$, and so $a=b+\delta_j-\delta_1$, where $b_U=b_V$ and $b_{\ell}\geq 0$ for each $\ell\in [m]$. By Lemma~\ref{cri}, either $b\in H$ or $b=c+k(\delta_1+\delta_{m})$ with $c\in H$ and $k>0$. The second case cannot happen because $a_1=-1$.  Hence  for any $e\in E(G)$, $a+h(e)\in H$ if and only if $1\in e$.  In other words, $a+h(e)\in H$ if and only if   $e\in \{\{1,k+1\},\{1,k+2\},\ldots,\{1,m-1\}\}$. Denote $\{1,k+i\}$ by $e_i$  for
$i=1,\ldots,m-k-1$. Let $C_i$ be the cycle $1,k+i,2,m-1$  and let $v_i=v(C_i)$ for $i=1,\ldots, m-k-2$. Then for $i=1, \ldots, m-k-2$, we have $v_i(i)\in\{\pm 1\}$ and $v_i(j)=0$ for $j\neq i$ and $j=1,\ldots, m-k-2$. This implies that $\dim_K D_a\geq m-k-2$. To compute $\dim_K KL_a$, we notice that if $C$ is an induced cycle with $1\notin V(C)$, then $a+h(v(C))\notin H$ and thus $KL_a$ contains the cycle space of the complete bipartite graph   with bipartition $U\setminus \{1\}$ and $V$, which has the dimension $(m-k)(k-1)-m+2$, see (\ref{dimension}). Thus $T^1(R)_a=0$ because $\dim_K KL=(m-k)k-m$.

Next we assume that $i=2$. Then $a=b+\delta_j-\delta_2$, where $b_{\ell}\geq 0$ for all $\ell\in [m]$  and $b_U=b_V$, $2\neq j\in U$. By Lemma~\ref{cri},  we have either $b\in H$ or $b=c+k(\delta_1+\delta_{m})$ for some $k\geq 1$ and with $c\in H$ and $\ell(c)=r(c)$. Suppose first that $b\notin H$ and $k\geq 2$. Then for any cycle $C$, $a+h(v(C))\in H$ implies $V(C)\cap \{1,m\}=\emptyset$. Thus, similarly as in Case 1  we see that $KL_a=KL$ and $T^1(R)_a=0$.
  Suppose next that $j\neq 1$ and  that $b\in H$ or $b\notin H$ and $k=1$. Then $a+h(e)\in H$ for any $e\in \{\{3,k+1\},\ldots, \{3,m-2\}\}$. Denote $\{3,k+t\}$ by $e_t$ for $t=1,\ldots, m-k-1$.  For $t=1,\ldots, m-k-1$, let $C_t$ be the cycle $3,k+t,4,m$ and let $v_t=v(C_t)$, the vector corresponding to  $C_t$. Then $v_t(t)\in \{\pm 1\}$ for $t=1,\ldots,m-k-1$ and $v_t(k)=0$ for $k\neq t$. This implies that $\dim_KD_a\geq m-k-1$.  On the other hand, $KL_a$ contains the cycle space of the subgraph of $G_{k,m-k}$ induced on  $\{1,3,4,\ldots, k\}\cup \{k+1,k+2,\ldots,m\}$, which has the dimension $(k-1)(m-k)-m+1$. Thus $T^1(R)_a=0$.

Finally suppose  that $j=1$ and that and also $k=1$ if  $b\notin H$.  If  $b\in H$, then  we check that $a+h(e)\in H$ for any $e\in \{\{2,k+1\},\ldots, \{3,m-2\}\}$ and deduce  that $T^1(R)_a=0$, in the same process as in the last case. If
$b\notin H$ and $k=1$, then for any induced cycle $C$, we have $a+h(v(C))\in H$ if and only if $2\in V(C)$ and $\{1,m\}\cap V(C)=\emptyset$. We claim that $KL_a=KL$. Given an induced cycle  $C: 2,i_1,i_2,i_3$ with $a+h(v(C))\in H$. Here $i_1$ and $i_3$ belong to $V$ and $i_2$ belong to U.  We let $C_1:2,i_1,i_2,m$ and $C_2: i_2,i_3,2,m$. Then $v(C_1)$ and $v(C_2)$ belong to $KL_a$ and $v(C)\in \{\pm v(C_1)\pm v(C_2)\}$. Thus $KL_a=KL$, as claimed. In particular, $T^1(R)_a=0$.

\medskip
\noindent
Case 3: $|\{k\:\; a_k<0\}|=2$. Without restriction we may assume $a_i=a_j=-1$ for some $i\neq j$.  Assume first that both $i$ and $j$ belong to $V$. Then for any induced cycle $C$ such that $\{i,j\}\nsubseteq V(C)$, we have $v(C)\in KL_a$. Let $C: k,i,\ell,j$ be a cycle with $\{i,j\}\subseteq V(C)$. We  choose   $d\in V\setminus \{i,j,m\}$. Then we obtain two cycles $C_1:$$ k,i,\ell,d$ and $C_2:$$ \ell,j,k,d$. Since $v(C)\in \{\pm v(C_1)\pm v(C_2)\}$  and since $v(C_t)\in KL_a$ for $t=1,2$, we have $v(C)\in KL_a$ and thus  $KL_a=KL$.  In particular $T^1(R)_a=0$.

  Next assume that  $i\in V$  and $j\in U$ and $\{i,j\}\neq \{1,m\}$. Notice that we can write $a$ as $a=b+k(\delta_1+\delta_{m})-(\delta_i+\delta_j)$, where $b\in H$ and $k\geq 0$. Moreover, if $k> 0$ then $\ell(b)=r(b)$ and $\{i,j\}\cap \{1,m\}=\emptyset$.

  If $k=0$, then $\dim_K D_a=1$,  and $KL_a$ contains the cycle space of the graph which is obtained from $G$ by deleting the edge $\{i,j\}$. Hence $\dim_KKL_a\geq \dim_KKL-1$, and so $T^1(R)_a=0$.

  If $k=1$, then for any induced cycle $C$, we have $a+h(v(C))\in H$ if and only if $\{i,j\}\subseteq V(C)$ and $V(C)\cap \{1,m\}=\emptyset$. Let $C:i,j,k,\ell$ be an induced cycle such that $a+h(v(C))\in H$. Then the vectors $v_1,v_2$ which correspond to cycles $j,k,\ell,m$ and $\ell,i,j,m$  belong to $KL_a$ and $v(C)\in \{\pm v(C_1)\pm v(C_2)\}$. It follows that $KL=KL_a$ and $T^1(R)_a=0$.

  If $k\geq 2$, then $\MF'_a=\emptyset$ by Lemma~\ref{cri} and it follows that  $KL=KL_a$. In particular $T^1(R)_a=0$.

 Finally assume that $\{i,j\}=\{1,m\}$. Then $a=b-\delta_1-\delta_{m}$ with $b\in H$ and so  $\MF'_a=\emptyset$.   It follows that $KL_a=KL$ and  $T^1(R)_a=0$.

\medskip
\noindent Case 4:  $|\{k\:\; a_{k}<0\}|=3$. We may assume that $a_i=a_j=a_k=-1$. We only need to  consider the case when $\MF'_a\neq \emptyset$. So we may assume $i,k$ belong to $V$ and $j\in U$, and $\{1,m\}\nsubseteq \{i,j,k\}$. Let $C:i,j,k,\ell$ be an induced cycle such that $a+h(v(C))\in H$. We choose  $d\in V\setminus \{i,k,m\}$ and let  $C_1: j,k,\ell,d$ and $C_2:\ell,i,j,d$ be two cycles in $G$.  Then $v(C_1)$ and $v(C_2)$ belong to $KL_a$ and $v(C)\in \{\pm v(C_1)\pm v(C_2)\}$. This implies $KL_a=KL$,  and in particular, $T^1(R)_a=0$.

\medskip
\noindent
Case 5:  $|\{k\:\; a_k<0\}|\geq 4$. If $|\{k\:\; a_k<0\}|= 4$, we may assume that $a_i=a_j=a_k=a_{\ell}=-1$. Then for any induced cycle $C$,  $a+h(v(C))\in H$ implies that $V(C)=\{ i,j,k,\ell\}$.  We may assume that $i$ and $k$ belong to $V$.  Choose  $t\in V\setminus \{j,\ell\}$, and let $C_1: i,j,k,t$ and $C_2: k,l,i,t$ be $4$-cycles of $G$. Since $v(C)\in\{\pm v(C_1)\pm v(C_2)\}$, we have  $KL_a=KL$, and consequently, $T^1(R)_a=0$. If $|\{k\:\; a_k<0\}|> 4$, then $\MF'_a=\emptyset$ and so $T^1(R)_a=0$.

Thus we have shown that $T^1(R)_a=$ for all $a\in \ZZ H$, and this shows that $R$ is rigid, as desired.
\end{proof}

The statement of Proposition~\ref{main4}   as well as its proof indicate that a graph $G$ which is obtained from the  complete bipartite $K(n,n)$ by removing $t$ edges is rigid if $n$ compared with $t$ is large.

\end{document}